\documentclass[a4paper,12pt]{amsart}
\usepackage{amsmath,amssymb,amsxtra,comment,graphicx,psfrag}
\usepackage{bm,mathrsfs}
\usepackage{mathtools}
\usepackage{xspace}
\usepackage{color}
\usepackage{xcolor}
\usepackage{hhline}

\definecolor{forestgreen}{rgb}{0.13, 0.55, 0.13}
\definecolor{lightblue}{rgb}{0.68, 0.85, 0.9}
\usepackage[colorlinks=true,
            linkcolor=blue,
            urlcolor  = blue,
            citecolor=forestgreen,
            anchorcolor = lightblue]{hyperref}

\usepackage{enumerate}

\usepackage{hyperref}
\usepackage{pdfsync}
\usepackage{hhline}
\usepackage{fancyhdr}
\usepackage{epsfig}
\usepackage{epsfig,subfigure,epstopdf}
\allowdisplaybreaks

\usepackage{pstricks,pst-plot,pst-func}
\usepackage{pspicture}
\usepackage{curves}

\include{rgb}

\def\N{{\mathbb N}}

\definecolor{mygray}{rgb}{0.9,0.9,0.9}

\begin{document}
\theoremstyle{plain}
\newtheorem{theorem}{Theorem}[section]
\newtheorem{lemma}{Lemma}[section]
\newtheorem{proposition}{Proposition}[section]
\newtheorem{corollary}{Corollary}[section]

\theoremstyle{definition}
\newtheorem{definition}[corollary]{Definition}

\newtheorem{example}{Example}[section]

\newtheorem{remark}{Remark}[section]
\newtheorem{remarks}[remark]{Remarks}
\newtheorem{note}{Note}
\newtheorem{case}{Case}

\numberwithin{equation}{section}
\numberwithin{table}{section}
\numberwithin{figure}{section}

\title[BDF3 method on variable grids]
{
 Variable step-size BDF3 method for Allen-Cahn equation}

\author[Minghua Chen]{Minghua Chen}
\address{School of Mathematics and Statistics, Gansu Key Laboratory of Applied Mathematics and Complex Systems,
 Lanzhou University, Lanzhou 730000, P.R. China}
\email {\href{mailto:chenmh@lzu.edu.cn}{chenmh{\it @\,}lzu.edu.cn}}

\author[Fan Yu]{Fan Yu}
\address{School of Mathematics and Statistics, Gansu Key Laboratory of Applied Mathematics and Complex Systems,
 Lanzhou University, Lanzhou 730000, P.R. China}
\email {\href{mailto:yuf20@lzu.edu.cn}{yuf20{\it @\,}lzu.edu.cn}}

\author[Qingdong Zhang]{Qingdong Zhang}
\address{School of Mathematics and Statistics, Gansu Key Laboratory of Applied Mathematics and Complex Systems,
 Lanzhou University, Lanzhou 730000, P.R. China}
 \email {\href{mailto:zhangqd19@lzu.edu.cn}{zhangqd19{\it @\,}lzu.edu.cn}}

\author[Zhimin Zhang]{Zhimin Zhang}
\address{Department of Mathematics, Wayne State University, Detroit, MI 48202, USA}
\email {\href{mailto:zmzhang@csrc.ac.cn}{zmzhang{\it @\,}csrc.ac.cn}}




\begin{abstract}

In this work, we analyze the three-step backward differentiation formula (BDF3) method
 for solving the Allen-Cahn equation on variable grids.
For BDF2 method, the discrete orthogonal convolution (DOC) kernels are positive, the stability and convergence analysis are well established in [Liao and  Zhang, \newblock  Math. Comp., \textbf{90} (2021) 1207--1226; Chen,  Yu, and  Zhang, \newblock SIAM J. Numer. Anal., Major  Revised]. However, the numerical analysis for BDF3 method with variable steps seems to be highly nontrivial, since the DOC kernels are not always positive.
By developing a novel spectral norm inequality,
the unconditional stability  and convergence are rigorously proved under the updated step ratio restriction $r_k:=\tau_k/\tau_{k-1}\leq 1.405$
(compared with $r_k\leq 1.199$ in [Calvo and Grigorieff, \newblock  BIT. \textbf{42} (2002) 689--701]) for BDF3 method.
Finally, numerical experiments are performed  to illustrate the theoretical results.
To the best of our knowledge, this is the first theoretical analysis of variable steps BDF3 method for the Allen-Cahn equation.
\end{abstract}

\maketitle


\section{Introduction}\label{Se:intro}
The  objective of this paper is to present a rigorous stability and convergence analysis of the BDF3 method with variable steps for solving the  Allen-Cahn equation \cite{Allen:79}
\begin{equation}\label{1.1}
\left \{
\begin{aligned}
&\partial_tu-\varepsilon^2\Delta u+f(u)=0, ~~(x,t)\in\Omega\times (0,T],\\
& u(0)=u_0,
\end{aligned}
\right .
\end{equation}
where the nonlinear bulk force is given by $f(u)=F'(u)=u^3-u$, and the parameter $\varepsilon>0$ represents the interface width.
For simplicity, we consider the periodic boundary conditions. The
above Allen-Cahn equation can be viewed as an $L^2$-gradient flow of the following
 free energy functional
\begin{equation}\label{1.2}
E[u]=\int_\Omega\left(\frac{\varepsilon^2}{2}\left|\nabla u\right|^2+F(u)\right)dx
~~{\rm with} ~~F(u)=\frac{1}{4}\left(u^2-1\right)^2.
\end{equation}
In other words, the Allen-Cahn equation \eqref{1.1} admits the  energy dissipation law:
\begin{equation}\label{1.3}
\frac{dE[u]}{dt}=-\int_\Omega\left|\partial_tu\right|^2dx\leq0.
\end{equation}

Let $N\in \N$ and choose the nonuniform time levels $0=t_0<t_1<\cdots<t_N=T$ with the time-step $\tau_k=t_k-t_{k-1}$ for $1\leq k\leq N$. For any time sequence $\{v^n\}_{n=0}^N$, denote
\begin{equation*}
\begin{split}
\nabla_{\tau}v^{n}:=v^n-v^{n-1},\quad\partial_{\tau}v^n:=\nabla_{\tau}v^{n}/\tau_n,~~n\geq 1.
\end{split}
\end{equation*}
For $k=1,2,3$, let $\Pi_{n,k}v$ denote the  Lagrange interpolating polynomial of a function $v$ over $k+1$ nodes  $t_n, t_{n-1}, \ldots, t_{n-k}$.
Define the adjacent time step ratio
\begin{equation*}
\begin{split}
r_k:=\frac{\tau_k}{\tau_{k-1}},~~k\geq 2.
\end{split}
\end{equation*}
Let  $v^n=v(t_n)$.
The BDF3 scheme  is defined by \cite{MCRDG,RDG84,GS,WR:08}
\begin{equation}\label{1.4}
\begin{split}
D_3v^n\!=\!\left(\Pi_{n,3}v\right)'(t_n)
\!=\!b_0^{(n)}\nabla_{\tau}v^{n}+
b_1^{(n)}\nabla_{\tau}v^{n-1}+b_2^{(n)}\nabla_{\tau}v^{n-2}\!=\!\sum_{k=1}^{n}b_{n-k}^{(n)}\nabla_\tau v^k,~~n\geq3.
\end{split}
\end{equation}
Here the coefficients are computed by
\begin{equation}\label{1.5}
\begin{split}
b_0^{(n)}&=\frac{\left(1+r_{n-1}\right)\left[1+2 r_n+r_{n-1}\left(1+4 r_n+3 r_n^2\right)\right]}{\tau_n\left(1+r_n\right)\left(1+r_{n-1}\right)\left(1+r_{n-1}+r_nr_{n-1}\right)},\\
b_1^{(n)}&=-\frac{r^{2}_n\left[(1+2r_{n-1}+r_nr_{n-1})^2-r_{n-1}(1+r_{n-1})\right]}
{\tau_n\left(1+r_n\right)\left(1+r_{n-1}\right)\left(1+r_{n-1}+r_nr_{n-1}\right)},\\
b_2^{(n)}&=\frac{r_n^{2}r^{3}_{n-1}\left(1+r_n\right)^2}{\tau_n\left(1+r_n\right)\left(1+r_{n-1}\right)\left(1+r_{n-1}+r_nr_{n-1}\right)}~~~~\quad~~~~ {\rm with}~~\quad~~ b_{j}^{(n)}=0,~~ j\geq 3.
\end{split}
\end{equation}

Since BDF3 scheme needs three starting values, for concreteness, we use BDF1 scheme and BDF2 scheme, respectively, to compute the first-level
solution $u^1$ and  second-level solution $u^2$, namely,
\begin{equation}\label{1.6}
\begin{split}
D_3v^1:=D_1v^1=\nabla_{\tau}v^{1}/\tau_1,~~~~D_3v^2:=D_2v^2=\frac{1+2 r_2}{\tau_2(1+r_2)}\nabla_{\tau}v^{2}-\frac{ r_2^2}{\tau_2(1+r_2)}\nabla_{\tau}v^{1}.
\end{split}
\end{equation}

We recursively define a sequence of approximations $u^n$ to
the nodal values $u(t_n)$ of the exact solution by BDF3 method
\begin{equation}\label{1.7}
D_3u^n-\varepsilon^2\Delta u^n+f(u^n)=0,~~n\geq1
\end{equation}
with the initial data $u^0=u_0$ and $f(u^n)=\left(u^n\right)^3-u^n$.


The BDF3 operator \eqref{1.4} and \eqref{1.6} are regarded as a discrete convolution summation
\begin{equation}\label{1.8}
\begin{split}
D_3v^n=\sum_{k=1}^{n}b_{n-k}^{(n)}\nabla_\tau v^k,~~n\geq1
\end{split}
\end{equation}
with $b_0^{(1)}=1/\tau_1,~~~~b_{0}^{(2)}=\frac{1+2r_2}{\tau_2(1+r_2)},~~~b_{1}^{(2)}=-\frac{ r_2^2}{\tau_2(1+r_2)}$ in \eqref{1.6}
and $b_{n-k}^{(n)}$ in \eqref{1.5}.

Following the approach of \cite{CYZ:21,LZ}, the discrete orthogonal convolution (DOC) kernels $\{d_{n-k}^{(n)}\}_{k=1}^n$ are defined by
\begin{equation}\label{1.9}
\begin{split}
d_{0}^{(n)}:&=\frac{1}{b_{0}^{(n)}}~~~~{\rm and}~~~~ d_{n-k}^{(n)}:=-\frac{1}{b_0^{(k)}}\sum_{j=k+1}^nd_{n-j}^{(n)}b_{j-k}^{(j)},~~1\leq k\leq n-1.
\end{split}
\end{equation}
Moreover, the DOC kernels $\{d_{n-k}^{(n)}\}_{k=1}^n$ satisfy the discrete orthogonal identity
\begin{equation}\label{1.10}
\begin{split}
\sum_{j=k}^n d_{n-j}^{(n)}b_{j-k}^{(j)}\equiv\delta_{nk},~~1\leq k\leq n.
\end{split}
\end{equation}

For convenience, we introduce the following matrices:
\begin{equation}\label{1.11}
B:=\begin{pmatrix}
b_0^{(1)}&&&&\\
b_1^{(2)}&b_0^{(2)}&&&\\
b_2^{(3)}&b_1^{(3)}&b_0^{(3)}&&\\
&\ddots&\ddots&\ddots&&\\
&&b_{2}^{(n)}&b_1^{(n)}&b_0^{(n)}
\end{pmatrix}
~~{\rm and}~~
D:=\begin{pmatrix}
d_0^{(1)}&&&&\\
d_1^{(2)}&d_0^{(2)}&&&\\
d_2^{(3)}&d_1^{(3)}&d_0^{(3)}&&\\
\vdots&\vdots&\ddots&\ddots&\\
d_{n-1}^{(n)}&d_{n-2}^{(n)}&\cdots&d_1^{(n)}&d_0^{(n)}
\end{pmatrix},
\end{equation}
where the discrete convolution kernels $b_{n-k}^{(n)}$ and the DOC kernels $\{d_{n-k}^{(n)}\}_{k=1}^n$ are defined in \eqref{1.8} and \eqref{1.9}, respectively. From the
discrete orthogonal identity \eqref{1.10}, there exists $D=B^{-1}$.

\begin{remark}
 For BDF2 method, the DOC kernels are positive \cite{CYZ:21,LZ}. However,  the DOC kernels are not always positive in \eqref{1.9} for BDF3 method, which implies   the numerical analysis seems to be highly nontrivial.
\end{remark}

The variable  time-stepping technique is powerful in capturing the multi-scale behaviors (e.g., the solution  changes rapidly in certain regions of time)
of phase fields model including the Allen-Cahn model.
Due to the strong stability (A-stable),  the numerical analysis of variable steps BDF2 method for ODEs and PDEs  receives much attentions including some early works \cite{Be,CL,E,T} and very recent works \cite{CYZ:21,CWYZ,LTZ,LZ}.
However, the numerical analysis for BDF3 method with variable steps seems to be highly nontrivial (compared with the BDF2 method).
As for variable steps BDF3 method for ODEs, Grigorieff et al. prove that it is zero-stable if the adjacent time-step ratio $r_k<1.08$ in \cite{RDG83}
and extend to $r_k<1.292$ in \cite{RDG84}, $r_k<1.462$  in \cite{MC90}.
Based on a spectral radius approach, Guglielmi and Zennaro prove the zero-stability of variable steps BDF3 method for $r_k<1.501$ in \cite{NG01}.
Variable steps implicit-explicit BDF3 method is presented by Wang and Ruuth \cite{WR:08},
where the zero-stability with $r_k<1.501$  is also proved for ODEs.
Recently, the stability is established under the adjacent time-step ratio $r_k<2.553$ of variable steps BDF3 method for ODE problem in \cite{LL}.
The stability of the variable steps BDF3 method for a  parabolic problem is derived by Calvo and Grigorieff  \cite{MCRDG} under the time-step  ratio $r_k\leq 1.199$.
However, it contains a factor  $\exp(C\Gamma_n)$ with $\Gamma_n=\sum_{k=2}^n|r_k-r_{k-1}|$, the quantity $\Gamma_n$ may be unbounded at vanishing step sizes for certain choices of time-steps.
We are unaware of any other published works on the stability analysis of the variable steps BDF3 method for a time-dependent PDEs.
In this paper,  the variable steps BDF3  scheme  is investigated to solve the Allen-Cahn equation.
For the BDF2 method, the associated DOC kernels are positive, the stability and convergence analysis are well established in \cite{CYZ:21,LZ}. However, such analysis technique can not apply to BDF3 method directly, since  the DOC kernels are not always positive in \eqref{1.9}.
By developing a novel spectral norm inequality,
the unconditional stability  and convergence are rigorously proved under the updated step ratio restriction $r_k\leq 1.405$ for BDF3 method, which greatly improve the adjacent time-step ratios $r_k\leq 1.199$ in   \cite{MCRDG}.
As far as we know, this is the first theoretical analysis of variable steps BDF3 method for the Allen-Cahn equation.

The rest of this paper is organized as follows.
In  the next Section, we show that the upper bound of the fixed adjacent time-step ratio is less than $\sqrt{3}$  in a sense of the positive semi-definiteness of the matrix $B$
in \eqref{1.11}.
In Section \ref{Se:variable ratio}, we  prove the variable adjacent time-step ratio $r_k\leq 1.405$, which plays an important role in our numerical analysis.
In Section \ref{Se:solv}, the unique solvability of the variable-steps BDF3 scheme \eqref{1.7} is established in Theorem \ref{Theorem:4.1} by using the fact that the solution of nonlinear scheme is equivalent to the minimization of a convex functional. A discrete energy stability is proved under the adjacent time-step ratio $r_k<1.405$  in Theorem \ref{Theorem:4.2}.
By developing a novel spectral norm inequality in Lemma \ref{Lemma:5.2},  the unconditional stability  and the convergence  of  BDF3 scheme \eqref{1.7} are rigorously proved in Section \ref{Se:stab}.
Finally, numerical experiments are carried out to corroborate the theoretical results.

\section{Estimate for fixed adjacent time-step ratio}\label{sec2}
In this section, we show that the upper bound of fixed adjacent time-step ratio is less than $\sqrt{3}$  in a sense of the positive semi-definiteness
 of the matrix $B$ in \eqref{1.11}.
First, we introduce some lemmas that will be used later.

\begin{proposition}\cite[p.\,28]{Quarteroni:07}\label{pr:2.1}
A matrix $P\in \mathbb{R}^{n\times n}$ is said to be positive definite in $\mathbb{R}^{n}$ if $(Px,x)>0$, $\forall x \in \mathbb{R}^{n}$, $x\neq 0$.
A real matrix $P$ of order $n$ is positive definite  if and only if  its symmetric part $H=\frac{P+P^T}{2}$ is positive definite.
\end{proposition}

\begin{definition}\cite[p.\,13]{Chan:07}\label{de:2.1}
Let  $n \times n$ Toeplitz  matrix  $T_n$ be of the following form
\begin{equation*}
T_n=\left [ \begin{matrix}
                      t_0           &      t_{-1}             &      \cdots         &       t_{2-n}       &       t_{1-n}      \\
                      t_{1}         &      t_{0}              &      t_{-1}         &      \cdots         &       t_{2-n}       \\
                     \vdots         &      t_{1}              &      t_{0}          &      \ddots         &        \vdots        \\
                     t_{n-2}        &      \cdots             &      \ddots         &      \ddots         &        t_{-1}         \\
                     t_{n-1}        &       t_{n-2}           &      \cdots         &       t_1           &        t_0
 \end{matrix}
 \right ];
\end{equation*}
i.e., $t_{i,j}=t_{i-j}$ and $T_n$ is constant along its diagonals. Assume that the diagonals $\{t_k\}_{k=-n+1}^{n-1}$ are the Fourier coefficients of a function
$g$, i.e.,
\begin{equation*}
\begin{split}
t_k=\frac{1}{2\pi}\int_{-\pi}^{\pi}g(x)e^{-ikx}dx.
\end{split}
\end{equation*}
Then the function $g(x)=\sum_{k=1-n}^{n-1} t_k e^{ikx}$ is called the generating function of $T_n$.
\end{definition}

\begin{lemma}\cite[p.\,13-15]{Chan:07}\label{Le:2.1} (Grenander-Szeg\"{o} theorem) Let $T_n$ be given by above matrix with a generating function $g$,
where $g$ is a $2\pi$-periodic continuous real-valued functions defined on $[-\pi,\pi]$.
Let $\lambda_{\min}(T_n)$ and $\lambda_{\max}(T_n)$ denote the smallest and largest eigenvalues of $T_n$, respectively. Then we have
\begin{equation*}
  g_{\min} \leq \lambda_{\min}(T_n) \leq \lambda_{\max}(T_n) \leq g_{\max},
\end{equation*}
where $g_{\min}$ and  $g_{\max}$  is the minimum and maximum values of $g(x)$, respectively.
Moreover, if $g_{\min}< g_{\max}$, then all eigenvalues of $T_n$ satisfies
\begin{equation*}
  g_{\min} < \lambda(T_n) < g_{\max},
\end{equation*}
for all $n>0$. In particular, if  $g_{\min}>0$, then $T_n$ is positive definite.
\end{lemma}
\begin{lemma}\cite[p.\,29]{Quarteroni:07}\label{Le:2.2}(Sylvester criterion)
Let $P \in \mathbb{R}^{n\times n}$ be symmetric. Then, $P$ is positive definite
if and only if the dominant principal minors of $P$ are  all positive.
\end{lemma}

\begin{lemma}\label{Le:2.3}
Let   $n \times n$   matrices  $K_{n\times n}$ and $L_{n\times n}$ with  $p_j\neq 0$ be of the following form
\begin{equation*}
K_{n\times n}=\begin{pmatrix}
a_1~&b_2~&c_3&&&\\
b_2~&a_2~&b_3&\ddots&&\\
c_3~&b_3~&a_3~&\ddots~&\ddots&\\
&\ddots&\ddots&\ddots&\ddots&c_n\\
&&\ddots&\ddots~&\ddots&b_n\\
&&&c_n&b_n&a_n
\end{pmatrix}
~~{\rm and}~~
L_{n\times n}=\begin{pmatrix}
p_1&q_2&c_3&&&\\
&p_2&q_3&\ddots&&\\
&&p_3&\ddots&\ddots&\\
&&&\ddots&\ddots&c_n\\
&&&&p_{n-1}&q_n\\
&&&&&p_n
\end{pmatrix}.
\end{equation*}
  Then  the dominant principal minors of $K$ are
$$\det K_{j\times j}=\det L_{j\times j}$$
with
\begin{equation*}
\begin{split}
&p_1=a_1,\quad q_2=b_2,\quad p_2=a_2-\frac{1}{p_1}q_2^2,\\
&q_j=b_j-\frac{q_{j-1}}{p_{j-2}}c_j~~~~{\rm and}~~~~p_j=a_j-\frac{1}{p_{j-2}}c_j^2-\frac{1}{p_{j-1}}q^2_j,~~~~ j\geq3.
\end{split}
\end{equation*}
\end{lemma}
\begin{proof}
 Using the elementary row operations, the desired result is obtained.
\end{proof}
\subsection{Ratio estimate by Grenander-Szeg\"{o} theorem}\label{2.1.1}
Let the diagonal matrix be
\begin{equation}\label{2.1}
\begin{split}
\Lambda={\rm diag}\left(\tau_1,\tau_2, \ldots, \tau_n  \right).
\end{split}
\end{equation}
From  \eqref{1.11}, we have 
\begin{equation}\label{2.2}
\begin{split}
A:=\Lambda^{1/2}B\Lambda^{1/2}
=\begin{pmatrix}
a_0^{(1)}&&&&\\
a_1^{(2)}&a_0^{(2)}&&&\\
a_2^{(3)}&a_1^{(3)}&a_0^{(3)}&&\\
&\ddots&\ddots&\ddots&\\
&&a_{2}^{(n)}&a_1^{(n)}&a_0^{(n)}
\end{pmatrix}.
\end{split}
\end{equation}
Here the coefficients are computed by
\begin{equation}\label{2.3}
\begin{split}
a_0^{(1)}&=1,~~~~a_{0}^{(2)}=\frac{1+2r_2}{1+r_2},~~~a_{1}^{(2)}=\frac{- r_2^{3/2}}{1+r_2},\\
a_0^{(n)}&=\frac{\left(1+r_{n-1}\right)\left[1+2 r_n+r_{n-1}\left(1+4 r_n+3 r_n^2\right)\right]}{\left(1+r_n\right)\left(1+r_{n-1}\right)\left(1+r_{n-1}+r_nr_{n-1}\right)},\\
a_1^{(n)}&=-\frac{r^{3/2}_n\left[\left(1+2r_{n-1}+r_nr_{n-1}\right)^2-r_{n-1}(1+r_{n-1})\right]}{(1+r_n)(1+r_{n-1})(1+r_{n-1}+r_nr_{n-1})},\\
a_2^{(n)}&=\frac{r_n^{3/2}r^{5/2}_{n-1}(1+r_n)^2}{(1+r_n)(1+r_{n-1})(1+r_{n-1}+r_nr_{n-1})},~~n\geq 3.
\end{split}
\end{equation}
We next estimate the upper adjacent time-step ratio in a sense of $r_{k-1}\equiv r_k$.
\begin{lemma}\label{Le:2.4}
Let (constant) adjacent time-step ratio satisfy  $r_k\leq  r_{\max}=1.716$, $k\geq 2$. Then the matrix $A$ defined in \eqref{2.2} or $B$ in \eqref{1.11} is positive semi-definite.
\end{lemma}
\begin{proof}
Let $y= r_k \leq r_{\max}=1.716$. We prove the desired result in the following two cases.

Case I: $r_k\leq 1$. Since
$$2a_0^{(1)}-\left|a_1^{(2)}\right|-\left|a_2^{(3)}\right|\geq 2-1-1/3=2/3,$$
and
\begin{equation*}
\begin{split}
2a_0^{(2)}-\left|a_1^{(2)}\right|-\left|a_1^{(3)}\right|-\left|a_2^{(4)}\right|=\frac{1}{(1+y)(1+y+y^2)}g(y)
\end{split}
\end{equation*}
with
\begin{equation*}
\begin{split}
g(y)&=2(1+2y)(1+y+y^2)-y^{3/2}(1+y+y^2)-y^{3/2}\left[(1+y)^3-y \right]-y^4-y^5\\
    &\geq 2(1+2y)(1+y+y^2)-y(1+y+y^2)-y\left[(1+y)^3-y \right]-y^4-y^5\\
    &=2+4y+3y^2-2y^4-y^5\geq 2.
\end{split}
\end{equation*}
Moreover, we have
\begin{equation*}
\begin{split}
2a_0^{(n)}-2\left|a_1^{(n)}\right|-2\left|a_2^{(n)}\right|&=\frac{2}{(1+y)(1+y+y^2)}g_1(y),~~n\geq 3
\end{split}
\end{equation*}
with
\begin{equation*}
\begin{split}
g_1(y)
&=(1+y)^3+y^2+2y^3-y^{3/2}(1+y)^3+y^{5/2}-y^4-y^5
\geq (1+y)^3(1-y)+y^2\geq 1.\\
\end{split}
\end{equation*}

From the above inequalities, we know that the symmetric matrix $A+A^T$ in \eqref{2.2} is a diagonally dominant matrix.
Using the Gerschgorin circle theorem, the eigenvalues of $A+A^T$ are greater than zero,
it implies that the matrix $A$ is positive definite by Proposition \ref{pr:2.1}.

Case II: $1< r_k \leq r_{\max}=1.716$.
For any real sequence $\{w_k\}_{k=1}^n$, it holds that
\begin{equation*}
\begin{split}
2\omega_k\sum_{j=1}^ka_{k-j}^{(k)}\omega_j
=&2a_0^{(k)}\omega_k^2+2a_1^{(k)}\omega_k\omega_{k-1}+2a_2^{(k)}\omega_{k}\omega_{k-2}\\
=&a_2^{(k)}\omega^2_k+a_2^{(k)}\left(\frac{a_1^{(k)}}{2a_2^{(k)}}\omega_k+\omega_{k-1}\right)^2-a_2^{(k)}\omega^2_{k-1}-a_2^{(k)}\left(\frac{a_1^{(k)}}{2a_2^{(k)}}\omega_{k-1}+\omega_{k-2}\right)^2\\
&+2\left(a_0^{(k)}-a_2^{(k)}-\frac{\left(a_1^{(k)}\right)^2}{8a_2^{(k)}}\right)\omega_k^2+a_2^{(k)}\left(\omega_k+\frac{a_1^{(k)}}{2a_2^{(k)}}\omega_{k-1}+\omega_{k-2}\right)^2
\end{split}
\end{equation*}
with
$$a_0^{(k)}=a_0^{(3)}, a_1^{(k)}=a_1^{(3)}, a_2^{(k)}=a_2^{(3)},~~ k\geq 3.$$
We can check that
\begin{equation}\label{2.4}
a_0^{(k)}-a_2^{(k)}-\frac{\left(a_1^{(k)}\right)^2}{8a_2^{(k)}}=\frac{y^3}{8a_2^{(k)}(1+y)^2(1+y+y^2)^2}l(y)>0~~{\rm with}~~1\leq y \leq 1.731<\sqrt{3},
\end{equation}
since
\begin{equation*}
\begin{split}
l(y)
&=-8y^7-17y^6+10y^5+43y^4+42y^3+22y^2+4y-1\\
&\geq y\left(10y^4 -8r_{\max}^3y^3-17r_{\max}^2y^3+43y^3+42y^2+22y+3 \right)>0.
\end{split}
\end{equation*}
Using \eqref{2.3} and the above equations,  there exists
\begin{equation*}
\begin{split}
&2\sum_{k=1}^n\omega_k\sum_{j=1}^ka_{k-j}^{(k)}\omega_j
 \geq \left( 2a_0^{(1)}-a_2^{(3)} \right)w_1^2\!+\!\left( 2a_1^{(2)}-a_1^{(3)} \right)w_1w_2
\!+\!\left( 2a_0^{(2)}\!-\!a_2^{(3)}\!-\!\frac{\left(a_1^{(3)}\right)^2}{4a_2^{(3)}} \right)w_2^2\\
&=\left(w_1,w_2\right)
\begin{pmatrix}
2a_0^{(1)}-a_2^{(3)}&a_1^{(2)}-\frac{a_1^{(3)}}{2}\\
a_1^{(2)}-\frac{a_1^{(3)}}{2}&2a_0^{(2)}-a_2^{(3)}-\frac{\left(a_1^{(3)}\right)^2}{4a_2^{(3)}}
\end{pmatrix}
\begin{pmatrix}
w_1\\
w_2
\end{pmatrix}\geq 0,
\end{split}
\end{equation*}
since the dominant principal minors of the above $2\times 2$ matrix are greater than zero if $1< r_k \leq r_{\max}=1.716$.
The proof is completed.
\end{proof}

\begin{remark}\label{re:2.1}
The generating function of BDF3 kernels $b_{n-k}^{(n)}$ in \eqref{1.4} is
\begin{equation}\label{2.5}
\begin{split}
g(x)&=a_0^{(n)}+a_1^{(n)}\cos\varphi+a_2^{(n)}\cos\left(2\varphi\right)
=2a_2^{(n)}x^2+a_1^{(n)}x+\left(a_0^{(n)}-a_2^{(n)}\right),~~n\geq 3
\end{split}
\end{equation}
with ${\cos\varphi=x}$,  $x\in[-1,1]$,  $\varphi\in[-\pi,\pi]$.
From   $2a_2^{(n)}>0$  and  \eqref{2.4}, it implies that $g(x)>0$ if $1< r_k \leq 1.731$.
Then  BDF3 kernels $b_{n-k}^{(n)}$ in \eqref{1.4} are  positive definite by Lemma \ref{Le:2.1}.
In fact, combining with  the proof process of Case I in Lemma \ref{Le:2.4},
the BDF3 kernels $b_{n-k}^{(n)}$ in \eqref{1.4} are  positive definite if $r_k \leq 1.731<\sqrt{3}$.
\end{remark}

\subsection{Ratio estimate by Sylvester criterion}\label{sec2.2}
From subsection \ref{2.1.1}, we know that the matrix  $A$ in \eqref{2.2} is  positive semi-definite with $r_s\leq 1.716$.
Moreover, the BDF3 kernels $b_{n-k}^{(n)}$ in \eqref{1.4} are  positive definite if $r_k \leq 1.731<\sqrt{3}$ in \eqref{2.5}.
In fact, we can check that the generating function  $g(x)<0 $ in \eqref{2.5} if $x=0.434$,  $r_k=1.732$.
However, the positive definiteness  with  $r_k=1.732$  in \eqref{1.4}  by Grenander-Szeg\"{o} theorem is still  in doubt.
Therefore, we need to look for the upper bound estimate  of other forms.

Let $A$ be given in \eqref{2.2}. From Lemma \ref{Le:2.3},
the dominant principal minors of $A+A^{T}$ are
\begin{equation*}
\begin{split}
\det \left(A+A^{T}\right)_{k\times k}=\det L_{k\times k}.
\end{split}
\end{equation*}
Here  the coefficients of $L_{k\times k}$ are
\begin{equation}\label{2.6}
\begin{split}
&p_1=2a_0^{(1)},\quad q_2=a_1^{(2)},\quad p_2=2a_0^{(2)}-\frac{1}{p_1}q_2^2,\\
&q_j=a_1^{(j)}-\frac{q_{j-1}}{p_{j-2}}a_2^{(j)}~~~~{\rm and}~~~~p_j=2a_0^{(j)}-\frac{1}{p_{j-2}}\left(a_2^{(j)}\right)^2-\frac{1}{p_{j-1}}q^2_j,~~~~ j\geq3.
\end{split}
\end{equation}
As a counterexample, we take  $r_s:=r_n=r_{n-1}<\sqrt{3}$ in \eqref{2.6}. According to Sylvester criterion in Lemma \ref{Le:2.2} and \eqref{2.6},
we know that there exists  a dominant principal minor of $A+A^{T}$ is negative, since there exists $p_j<0$, see Figure \ref{Sylvester}.
Hence, the matrix $A$ in \eqref{2.2} or $B$ in \eqref{1.11} is not positive definite if $r_s=1.732<\sqrt{3}$.
\begin{figure}[htb]
  \centering
  \includegraphics[width=6cm]{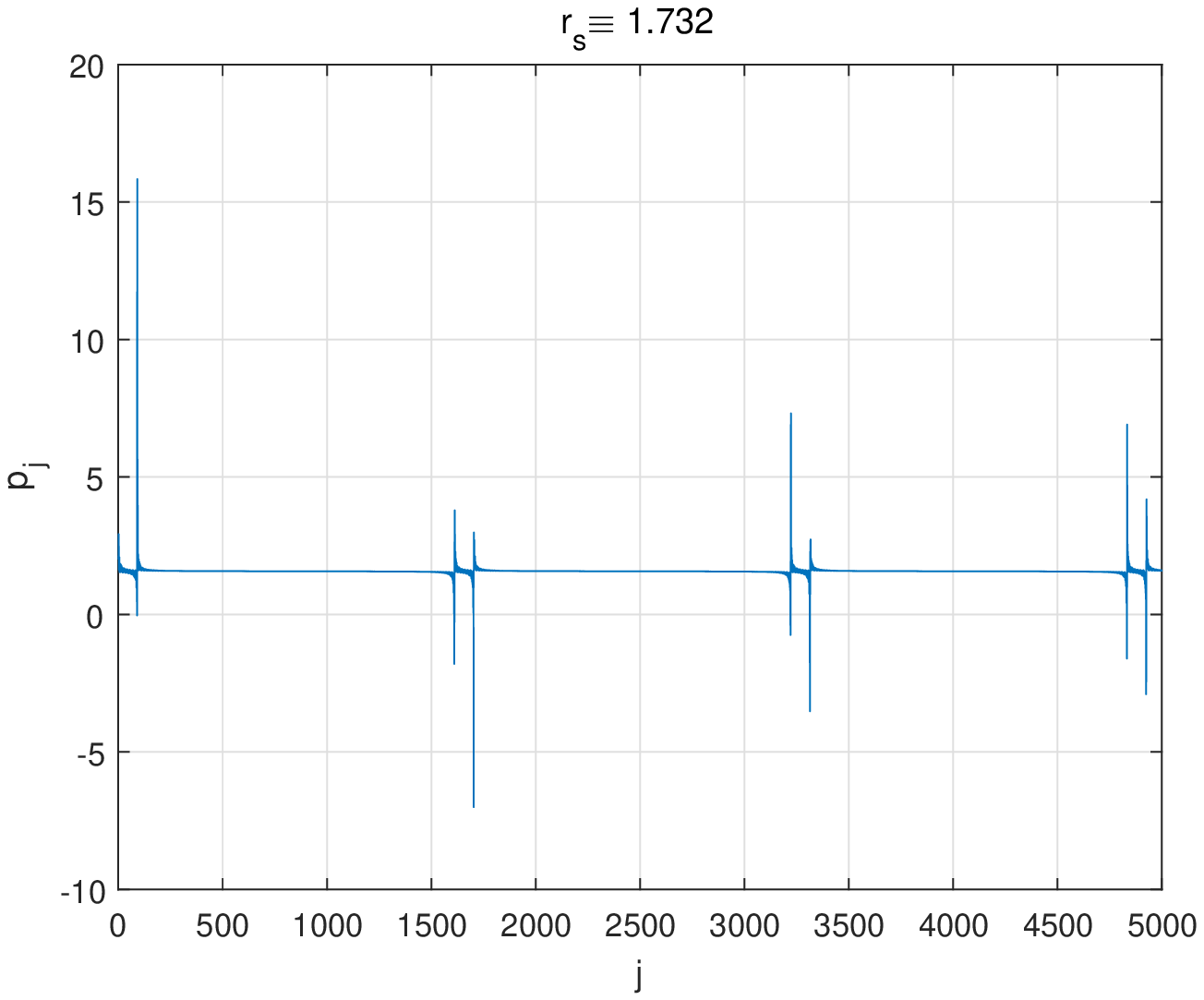}~~
  \includegraphics[width=6cm]{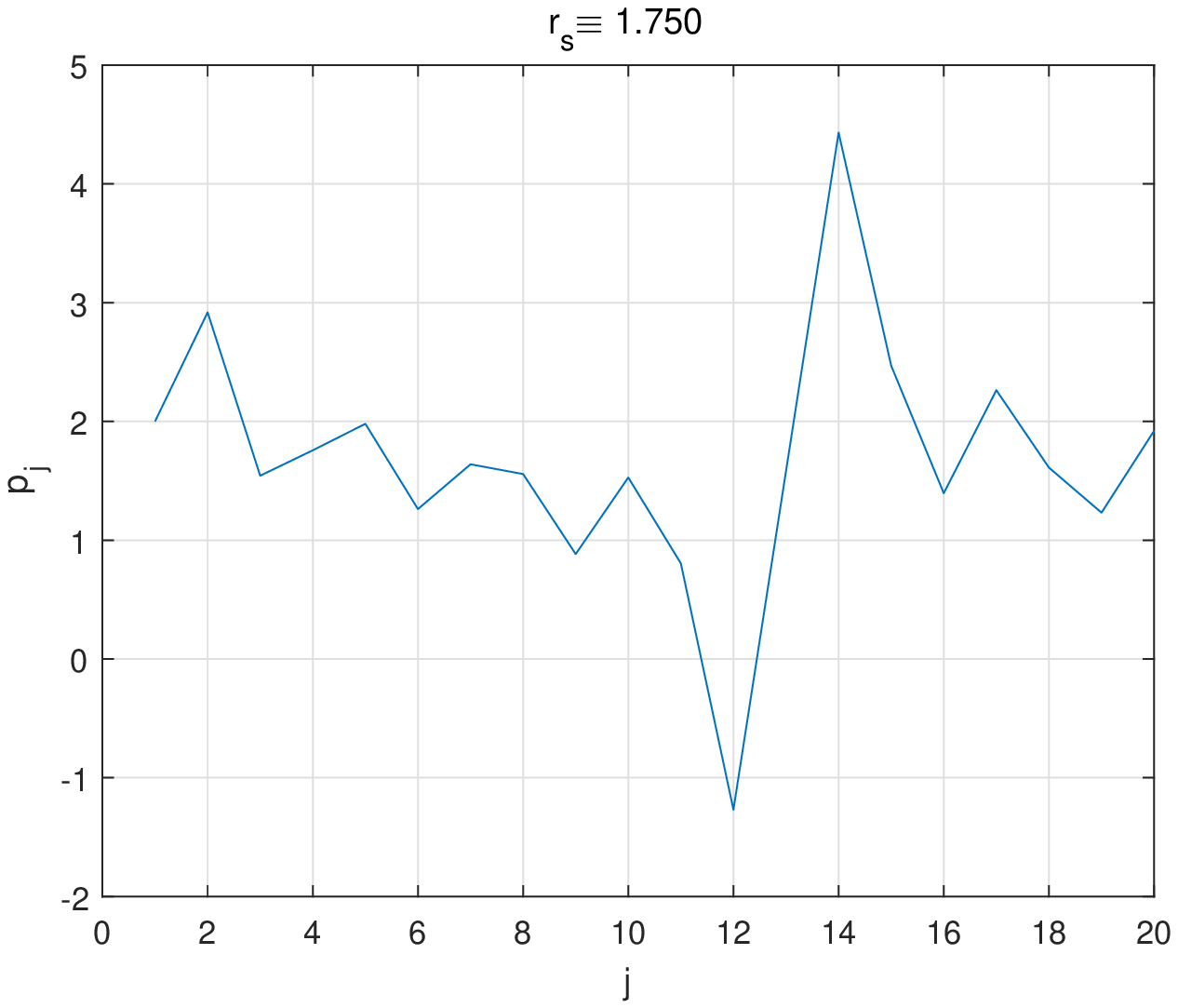}~~
  \caption{\small The graphs of  $p_j$ in \eqref{2.6}, left and right, respectively, for $r_s\equiv 1.732$ and $r_s\equiv 1.750$.   }  \label{Sylvester}
\end{figure}

\section{Estimate of variable  adjacent time-step ratio}\label{Se:variable ratio}
From  Section \ref{sec2}, the upper bound of fixed adjacent time-step ratio is less than $\sqrt{3}$  in a sense of the positive semi-definiteness
 of the matrix $B$ in \eqref{1.11}.
However, it is not easy to obtain an sharp  estimate for the general (variable) adjacent time-step ratio.
In this section, we  prove the  variable adjacent time-step ratio $r_k\leq 1.405$,
which plays an important role in our numerical analysis for the variable steps BDF3 scheme \eqref{1.7}.

Let $B$ and  $\Lambda$ be given in \eqref{1.11} and \eqref{2.1}, respectively. Let
\begin{equation}\label{3.1}
\begin{split}
\widetilde{B}=B-\gamma\Lambda^{-1}~~{\rm with}~~\gamma=1/200.
\end{split}
\end{equation}
From Lemma \ref{Le:2.3}, the dominant principal minors of $\widetilde{B}+\widetilde{B}^{T}$  are
\begin{equation*}
\begin{split}
\det \left(\widetilde{B}+\widetilde{B}^{T}\right)_{j\times j}=\det L_{j\times j}.
\end{split}
\end{equation*}
Here the coefficients of $L$ are computed by
\begin{equation}\label{3.2}
\begin{split}
&p_1=\widehat{b}_0^{(1)},\quad q_2=b_1^{(2)},\quad p_2=\widehat{b}_0^{(2)}-\frac{1}{p_1}q_2^2,\\
&q_j=b_1^{(j)}-\frac{q_{j-1}}{p_{j-2}}b_2^{(j)}~~~~{\rm and}~~~~p_j=\widehat{b}_0^{(j)}-\frac{1}{p_{j-2}}\left(b_2^{(j)}\right)^2-\frac{1}{p_{j-1}}q^2_j,~~j\geq3
\end{split}
\end{equation}
with
\begin{equation}\label{3.3}
\begin{split}
\widehat{b}_0^{(1)}&=1.99/\tau_1,~~~~\widehat{b}_{0}^{(2)}=\frac{1.99+3.99r_2}{\tau_2(1+r_2)},~~~b_{1}^{(2)}=\frac{- r_2^2}{\tau_2(1+r_2)},\\
\widehat{b}_0^{(n)}&=\frac{\left(1+r_{n-1}\right)\left[1.99+3.99 r_n+r_{n-1}\left(1.99+7.98 r_n+5.99 r_n^2\right)\right]}{\tau_n(1+r_n)(1+r_{n-1})(1+r_{n-1}+r_nr_{n-1})},\\
b_1^{(n)}&=-\frac{r^{2}_n\left[(1+2r_{n-1}+r_nr_{n-1})^2-r_{n-1}(1+r_{n-1})\right]}{\tau_n(1+r_n)(1+r_{n-1})(1+r_{n-1}+r_nr_{n-1})},\\
b_2^{(n)}&=\frac{r_n^{2}r^{3}_{n-1}(1+r_n)^2}{\tau_n(1+r_n)(1+r_{n-1})(1+r_{n-1}+r_nr_{n-1})},~~n\geq 3.
\end{split}
\end{equation}
The main aim of this part is to estimate  the following  inequality
\begin{equation*}
\begin{split}
\frac{\lambda_{\min}}{\tau_j}\leq  p_j\leq \frac{\lambda_{\max}}{\tau_j},~~j\geq 1 ~~{\rm with}~~\lambda_{\min}=1.99,~~\lambda_{\max}=3.99,
\end{split}
\end{equation*}
it implies that the matrix $A$ in \eqref{2.2} or $B$ in \eqref{1.11} is  positive definite by Sylvester criterion.

\subsection{A few technical lemmas}
First, we give some lemmas that will be used later.
\begin{lemma}\label{Le:3.1}
Let  $\Psi(x,y)$ with $(x,y)\in[0,r_s]\times[0,r_s]$, $r_s=1.405$ be defined by
\begin{equation*}
\begin{split}
\Psi(x,y)=\frac{(1+2y+xy)^2-y(1+y)}{(1+y)(1+y+xy)}-\frac{\kappa y^4(1+x)^2}{(1+y)^2(1+y+xy)}.
\end{split}
\end{equation*}
Here the coefficients are defined by
\begin{equation*}
\begin{split}
\kappa\in [\kappa_{\min},\kappa_{\max}],~~\kappa_{\min}=0.25, \kappa_{\max}=1.4.
\end{split}
\end{equation*}
Then we have
\begin{equation*}
\begin{split}
1 \leq \Psi(x,y)\leq2.7.
\end{split}
\end{equation*}
\end{lemma}
\begin{proof}
We can check that
\begin{equation*}
\begin{split}
\Psi(x,y)
&=1+\frac{\left(y^2+y^3-\kappa y^4\right)(1+x)^2+(1+x)y(1+y)^2}{(1+y)^2(1+y+xy)}\\
&=1+\frac{\left(y^2+\frac{3}{2}y^3-\kappa y^4\right)(1+x)^2+(1+x)y(1+y)^2-\frac{1}{2}y^3(1+x)^2}{(1+y)^2(1+y+xy)}\\
&=1+\frac{\left(y^2+\frac{3}{2}y^3-\kappa y^4\right)(1+x)^2+\frac{1}{2}(1-x^2)y^3+(2+2x)y^2+(1+x)y}{(1+y)^2(1+y+xy)}\geq 1.
\end{split}
\end{equation*}
Here we use
\begin{equation*}
\begin{split}
y^2+\frac{3}{2}y^3-\kappa y^4=y^2\left(1+\frac{3}{2}y-\kappa y^2\right)\geq y^2\left(1-\frac{3}{2}y(y-1)\right)\geq 0,
\end{split}
\end{equation*}
and
\begin{equation*}
\begin{split}
&\frac{1}{2}\left(1-x^2\right)y^3+(2+2x)y^2+(1+x)y\\
&\quad=(1+x)y\left(  \frac{1}{2}(1-x)y^2+2y+1  \right)
\geq (1+x)y\left(  -0.3y^2+2y+1  \right)\geq 0.
\end{split}
\end{equation*}
Similarly, we have
\begin{equation*}
\begin{split}
\Psi(x,y)
&=2.7+\frac{(y^3-y^2-\kappa y^4)(1+x)^2+\Psi_1(x,y)}{(1+y)^2(1+y+xy)},
\end{split}
\end{equation*}
where the quadratic function $\Psi_1(x,y)$ is
\begin{equation*}
\begin{split}
\Psi_1(x,y)=&-1.7(1+y)^3+y(1+y)^2-0.7xy(1+y)^2+2y^2(1+x)^2\\
\leq&-1.7(1+y)^2\left(y/1.405+y\right)+y(1+y)^2-0.7xy(1+y)^2+4.81y^2(1+x)\\
\leq&-2.909y(1+y)^2+y(1+y)^2-0.7xy(1+y)^2+4.81y^2(1+x)\\
=&y\Psi_2(x,y)
\end{split}
\end{equation*}
with
$$\Psi_2(x,y)=-\left(1.909+0.7x\right)y^2+\left(0.992+3.41x\right)y-1.909-0.7x.$$
Since the discriminant of root formulas of $\Psi_2(x,y)$ is
\begin{equation*}
\begin{split}
\Delta&=(0.992+3.41x)^2-4(1.909+0.7x)(1.909+0.7x)<0,\\
\end{split}
\end{equation*}
which implies  $\Psi_1(x,y)\leq 0$. Moreover, we have
\begin{equation*}
\begin{split}
y^3-y^2-\kappa y^4=y^2(-\kappa y^2+y-1)\leq y^2(-0.25y^2+y-1)\leq 0.
\end{split}
\end{equation*}
 The proof is completed.
\end{proof}

\begin{lemma}\label{Le:3.2}
Let  $\psi(x,y)$ with $(x,y)\in[0,r_s]\times[0,r_s]$, $r_s=1.405$ be defined by
\begin{equation*}
\begin{split}
\psi(x,y)=-\frac{x^2\left[(1+2y+xy)^2-y(1+y)\right]}{(1+x)(1+y)(1+y+xy)}
+\frac{\kappa_{\max}x^2y^4(1+x)}{(1+y)^2(1+y+xy)}~~{\rm with}~~\kappa_{\max}=1.4.
\end{split}
\end{equation*}
Then we have $\psi(x,y)\leq  0$.
\end{lemma}
\begin{proof}
We can check
\begin{equation*}
\begin{split}
\psi(x,y)=\frac{x^2}{(1+x)(1+y)^2(1+y+xy)}\psi_1(x,y)
\end{split}
\end{equation*}
with
$$\psi_1(x,y)=-(1+y)^3-2y(1+x)(1+y)^2-y^2(1+x)^2(1+y)+y(1+y)^2+1.4y^4(1+x)^2.$$
Using the above equation and  $1.4y-2<0$, it yields
\begin{equation*}
\begin{split}
\psi_1(x,y)&\leq -2y(1+x)(1+y)^2-y^2(1+x)^2(1+y)+y(1+y)^2+1.4y^3(1+y)(1+x)^2\\
&= y(1+y)\left[  -2(1+x)(1+y)+y(1+x)^2+(1+y)+(1.4y-2)y(1+x)^2  \right]\\
&\leq y(1+y)\psi_2(x,y)
\end{split}
\end{equation*}
with
\begin{equation*}
\begin{split}
\psi_2(x,y)&=  -2(1+x)(1+y)+y(1+x)^2+(1+y).
\end{split}
\end{equation*}
Since the first derivative of  $\psi_2(x,y)$ with respect to $y$ is greater than zero.
Then we have $\psi_2(x,y)\leq \psi_2(x,r_s)<0.$ The proof is completed.
\end{proof}

\begin{lemma}\label{Le:3.3}
Let  $\Phi(x,y)$ with $(x,y)\in[0,r_s]\times[0,r_s]$, $r_s=1.405$ be defined by
\begin{equation*}
\begin{split}
\Phi(x,y)=&\left[1.99+3.99 x+y\left(1.99+7.98 x+5.99 x^2\right)\right](1+x)(1+y)^4(1+y+xy)\\
&-\lambda_{\min}(1+x)^2(1+y)^4(1+y+xy)^2-\frac{1}{\lambda_{\min}}x^{3}y^{5}(1+x)^4(1+y)^2\\
&-\frac{1}{\lambda_{\min}}x^3\left[\left(1+2y+2xy\right)\left(1+y\right)^2+y^2\left(1+x\right)^2\left(1+y-\kappa_{\min}y^2\right)\right]^2
\end{split}
\end{equation*}
with $\lambda_{\min}=1.99$, $\kappa_{\min}=0.25$.
Then we have $\Phi(x,y)\geq  0$.
\end{lemma}
\begin{proof}
According to
\begin{equation*}
\begin{split}
1\leq 1+y-\kappa_{\min}y^2\leq 1+r_s-\kappa_{\min}r_s^2<1.912,
\end{split}
\end{equation*}
and
\begin{equation*}
\begin{split}
1.99+3.99 x+y\left(1.99+7.98 x+5.99 x^2\right)-\lambda_{\min}(1+x)(1+y+xy)=2x\left(1+2 y+2 xy\right),
\end{split}
\end{equation*}
it leads to
\begin{equation*}
\begin{split}
\lambda_{\min}\Phi(x,y)\geq& 3.98x\left(1+2 y+2 xy\right)(1+x)(1+y)^4(1+y+xy)-x^{3}y^5(1+x)^4(1+y)^2\\
&-x^3\left[\left(1+2y+2xy\right)(1+y)^2+1.912y^2(1+x)^2\right]^2\\
=&x(1+y)^4\Phi_1(x,y)+xy(1+x)(1+y)^2\Phi_2(x,y)+xy^2(1+x)^2\Phi_3(x,y).
\end{split}
\end{equation*}
Here
\begin{equation*}
\begin{split}
\Phi_1(x,y)&=  -1.13x^2y(1+x)-x^2+3.98x+3.98;\\
\Phi_2(x,y)&=11.94(1+x)(1+y)^2   -2.87x^2(1+y)^2\\
&\quad-2.793x^2y(1+x)(1+y)^2-3.824x^2 y(1+x); \\
\Phi_3(x,y)&=7.96(1+x)(1+y)^4-1.207x^2(1+y)^4-x^2y^3(1+x)^2(1+y)^2\\
&\quad-7.648x^2y(1+x)(1+y)^2-  3.655744x^2y^2(1+x)^2.
\end{split}
\end{equation*}
We can check
$\Phi_1(x,y)\geq   -1.13x^2r_s(1+r_s)-x^2+3.98x+3.98>0$ and
\begin{equation*}
\begin{split}
\Phi_2(x,y)&\geq11.94(1+x)(1+y)^2-2.87x^2(1+y)^2-6.718x^2y(1+y)^2-9.197x^2y\\
&=x(1+y)^2\left(11.94- 6.048xy \right)+h(x,y)\geq h(x,y)
  \end{split}
\end{equation*}
with $h(x,y)=11.94(1+y)^2-2.87x^2(1+y)^2-0.67x^2y(1+y)^2-9.197x^2y$.
Since the first derivative of  $h(x,y)$ with respect to $x$   is less than zero. It implies that $h(x,y)\geq h(r_s,y)$.
Moreover, the first derivative of  $h(r_s,y)$ with respect to $y$ is also less than zero. Then
\begin{equation*}
\begin{split}
\Phi_2(x,y)\geq h(x,y)\geq h(r_s,y)\geq h(r_s,r_s)>0.
\end{split}
\end{equation*}
On the other hand, there exists
\begin{equation*}
\begin{split}
\Phi_3(x,y)\geq xg(x,y)
\end{split}
\end{equation*}
with
\begin{equation*}
\begin{split}
g(x,y)&= 7.96\left(\frac{1}{r_s}+1\right)(1+y)^4-1.207x(1+y)^4-xy^3(1+r_s)^2(1+y)^2\\
&\quad-7.648xy(1+r_s)(1+y)^2-  3.655744xy^2(1+r_s)^2.
\end{split}
\end{equation*}
Since the first derivative of  $g(x,y)$ with respect to $x$   is less than zero. It implies that $g(x,y)\geq g(r_s,y)$. Furthermore, we have $g(r_s,y)\geq g(r_s,r_s)>0$.
Hence, there exists
\begin{equation*}
\begin{split}
\Phi_3(x,y)\geq xg(x,y)\geq xg(r_s,y)\geq xg(r_s,r_s)\geq 0.
\end{split}
\end{equation*}
The proof is completed.
\end{proof}

\begin{lemma}\label{Le:3.4}
Let  $\phi(x,y)$ with $(x,y)\in[0,r_s]\times[0,r_s]$, $r_s=1.405$ be defined by
\begin{equation*}
\begin{split}
\phi(x,y)=&\left[1.99+3.99 x+y\left(1.99+7.98 x+5.99 x^2\right)\right](1+x)(1+y)^4(1+y+xy)\\
&-\lambda_{\max}(1+x)^2(1+y)^4(1+y+xy)^2-\frac{1}{\lambda_{\max}}x^{3}y^{5}(1+x)^4(1+y)^2\\
&-\frac{1}{\lambda_{\max}}x^3\left[\left(1+2y+2xy\right)\left(1+y\right)^2+y^2\left(1+x\right)^2\left(1+y-\kappa_{\max}y^2\right)\right]^2
\end{split}
\end{equation*}
with $\lambda_{\max}=3.99$, $\kappa_{\max}=1.4$. Then we have $\phi(x,y)\leq  0$.
\end{lemma}
\begin{proof}
Using
\begin{equation*}
\begin{split}
&\left[\left(2y+2xy\right)(1+y)^2+y^2\left(1+x\right)^2\left(1+y-\kappa_{\max}y^2\right)\right]-2y^2(1+x)(1+y)\\
&\quad \geq y(1+x) \left[2(1+y)^2+y\left(1+y-\kappa_{\max}y^2\right)-2y(1+y)\right]\geq 0,
\end{split}
\end{equation*}
and
$$1.99+3.99 x+y\left(1.99+7.98 x+5.99 x^2\right)-\lambda_{\max}(1+x)(1+y+xy)=-2-2y+2x^2y,$$
it yields
\begin{equation*}
\begin{split}
\lambda_{\max}\phi(x,y)
&\leq\lambda_{\max}\left(-2-2y+2x^2y\right)(1+x)(1+y)^4(1+y+xy)-4x^3y^4\left(1+x\right)^2\left(1+y\right)^2\\
&=(1+x)^2(1+y)^2y\phi_1(x,y)+(1+x)(1+y)^4\phi_2(x,y).
\end{split}
\end{equation*}
Here the functions $\phi_1(x,y)$ and $\phi_2(x,y)$ are, respectively, defined by
\begin{equation*}
\begin{split}
\phi_1(x,y)=&7.98x^2y(1+y)^2-7(1+y)^3-4x^3y^3,
\end{split}
\end{equation*}
and
\begin{equation*}
\begin{split}
\phi_2(x,y)=&7.98x^2y-7.98(1+y)-0.98y(1+y)(1+x)\leq \phi_3(x,y)
\end{split}
\end{equation*}
with
$$\phi_3(x,y)=7.98x^2y-7.98(1+y)-0.98y(1+y).$$
Since the first derivative of  $\phi_1(x,y)$ and $\phi_3(x,y)$ with respect to $x$ is greater than zero. Hence
\begin{equation*}
\begin{split}
&\phi_1(x,y)\leq\phi_1(r_s,y)\leq\phi_1(r_s,r_s)\leq0,\\
&\phi_3(x,y)\leq\phi_3(r_s,y)\leq\phi_3(r_s,r_s)\leq0.
\end{split}
\end{equation*}
The proof is completed.
\end{proof}

\subsection{Estimate for variable time-step ratio by  Sylvester criterion}
We next prove the matrix $A$ in \eqref{2.2} or $B$ in \eqref{1.11} is  positive definite by Sylvester criterion.
\begin{lemma}\label{Le:3.5}
For any  adjacent time-step ratios $0<r_k\leq r_s=1.405$, $k\geq 2$, there exists
\begin{equation}\label{3.4}
 b_1^{(j)}+\mu_j\leq q_j \leq b_1^{(j)}+\nu_j\leq 0,~~~~j\geq 3
\end{equation}
with
\begin{equation}\label{3.5}
\begin{split}	
\mu_j=\frac{\kappa_{\min}r_j^2r_{j-1}^4(1+r_j)}{\tau_j(1+r_{j-1})^2(1+r_{j-1}+r_jr_{j-1})},~~ \nu_j=\frac{\kappa_{\max}r_j^2r_{j-1}^4(1+r_j)}{\tau_j(1+r_{j-1})^2(1+r_{j-1}+r_jr_{j-1})},
\end{split}
\end{equation}
and
\begin{equation}\label{3.6}
\frac{\lambda_{\min}}{\tau_j}\leq  p_j\leq \frac{\lambda_{\max}}{\tau_j},~~~~j\geq 1.
\end{equation}
Here the coefficients are defined by
$$\kappa_{\min}=0.25, \kappa_{\max}=1.4~~{\rm and}~~\lambda_{\min}=1.99, \lambda_{\max}=3.99. $$
\end{lemma}
\begin{proof}
From \eqref{3.2} and \eqref{3.3}, we obtain $p_1=\frac{1.99}{\tau_1}$, $q_2=\frac{- r_2^2}{\tau_2(1+r_2)}$ and
$$\frac{\lambda_{\max}}{\tau_2}\geq \frac{\lambda_{\max}}{\tau_2}-\frac{2+2r_2+\frac{1}{1.99}r_2^3}{\tau_2\left(1+r_2\right)^2}=p_2=\frac{\lambda_{\min}}{\tau_2}+\frac{2r_2+2r_2^2-\frac{1}{1.99}r_2^3}{\tau_2\left(1+r_2\right)^2}\geq\frac{\lambda_{\min}}{\tau_2}.$$
Next we prove \eqref{3.4} and \eqref{3.6} by mathematical induction.

For $j=3,$ using Lemma \ref{Le:3.2}, we have
\begin{equation}\label{3.7}
b_1^{(3)}+\mu_3 \leq b_1^{(3)}+\frac{\frac{1}{1.99}r_2^4r_3^2(1+r_3)}{\tau_3(1+r_2)^2(1+r_2+r_2r_3)}=q_3\leq  b_1^{(3)}+\nu_3=\psi(r_3,r_2)\leq 0.
\end{equation}
According to \eqref{3.2}, \eqref{3.3} and \eqref{3.7}, it yields
\begin{equation*}
\begin{split}
p_3&\geq \widehat{b}_0^{(3)}-\frac{\tau_1}{\lambda_{\min}}\left(b^{(3)}_2\right)^2-\frac{\tau_2}{\lambda_{\min}}\left(b_1^{(3)}+\mu_3\right)^2\\
&=\frac{\lambda_{\min}}{\tau_3}+\frac{1}{\tau_3(1+r_3)^2(1+r_{2})^4(1+r_{2}+r_2r_3)^2}\cdot \Phi(r_3,r_2)\geq \frac{\lambda_{\min}}{\tau_3}
\end{split}
\end{equation*}
with  $\Phi(r_3,r_2)\geq 0$ in Lemma \ref{Le:3.3}.

On the other hand, using  \eqref{3.2},  \eqref{3.3} and \eqref{3.7},  one has
\begin{equation*}
\begin{split}
p_3&\leq\widehat{ b}_0^{(3)}-\frac{\tau_1}{\lambda_{\max}}\left(b^{(3)}_2\right)^2-\frac{\tau_2}{\lambda_{\max}}\left(b_1^{(3)}+\nu_3\right)^2\\
&=\frac{\lambda_{\max}}{\tau_3}+\frac{1}{\tau_3(1+r_3)^2(1+r_{2})^4(1+r_2+r_2r_3)^2}\cdot \phi(r_3,r_2)\leq \frac{\lambda_{\max}}{\tau_3}
\end{split}
\end{equation*}
with  $\phi(r_3,r_2)\leq  0$  in Lemma \ref{Le:3.4}.

Supposing that \eqref{3.4} and \eqref{3.6} hold for $j=4,\ldots n-1$, namely,
\begin{equation}\label{3.8}
b_1^{(j)}+\mu_j\leq q_j\leq b_1^{(j)}+\nu_j\leq 0~~{\rm and}~~\frac{\lambda_{\min}}{\tau_j}
\leq p_j\leq \frac{\lambda_{\max}}{\tau_j},~~4\leq j\leq n-1.
\end{equation}
According to  \eqref{3.2}, \eqref{3.3}, \eqref{3.8} and Lemma \ref{Le:3.2}, there exists
\begin{equation*}
\begin{split}
q_n&=b_1^{(n)}-\frac{q_{n-1}}{p_{n-2}}b_2^{(n)}
\leq b_1^{(n)}+\frac{\tau_{n-2}}{\lambda_{\min}}\left(-q_{n-1}\right)b_2^{(n)} \leq b_1^{(n)}+\frac{\tau_{n-2}}{\lambda_{\min}}\left(-b_1^{(n-1)}-\mu_{n-1}\right)b_2^{(n)}\\
&= b_1^{(n)}+\frac{\Psi(r_{n-1},r_{n-2})}{\lambda_{\min}}\frac{r_n^{2}r^{4}_{n-1}(1+r_n)}{\tau_n(1+r_{n-1})^2(1+r_{n-1}+r_nr_{n-1})}\\
&\leq b_1^{(n)}+\nu_n=\psi(r_n,r_{n-1})\leq 0,
\end{split}
\end{equation*}
where $\Psi(r_{n-1},r_{n-2})$ and $\nu_n$ are, respectively, defined by Lemma \ref{Le:3.1} and \eqref{3.5}.

On the other hand, using   \eqref{3.2}, \eqref{3.3} and  \eqref{3.8}, we have
\begin{equation*}
\begin{split}
q_n&=b_1^{(n)}-\frac{q_{n-1}}{p_{n-2}}b_2^{(n)} \geq b_1^{(n)}+\frac{\tau_{n-2}}{\lambda_{\max}}(-q_{n-1})b_2^{(n)} \geq b_1^{(n)}+\frac{\tau_{n-2}}{\lambda_{\max}}\left(-b_1^{(n-1)}-\nu_{n-1}\right)b_2^{(n)}\\
&= b_1^{(n)}+\frac{\Psi(r_{n-1},r_{n-2})}{\lambda_{\max}}\frac{r_n^{2}r^{4}_{n-1}(1+r_n)}{\tau_n(1+r_{n-1})^2(1+r_{n-1}+r_nr_{n-1})} \geq b_1^{(n)}+\mu_n,
\end{split}
\end{equation*}
where $\Psi(r_{n-1},r_{n-2})$ and $\mu_n$ are, respectively, defined by Lemma \ref{Le:3.1} and \eqref{3.5}.

From  \eqref{3.2}, \eqref{3.3} and  \eqref{3.8}, it yields
\begin{equation*}
\begin{split}
p_n&=\widehat{b}_0^{(n)}-\frac{1}{p_{n-2}}\left(b^{(n)}_2\right)^2-\frac{1}{p_{n-1}}q^2_n 
\geq b_0^{(n)}-\frac{\tau_{n-2}}{\lambda_{\min}}\left(b^{(n)}_2\right)^2-\frac{\tau_{n-1}}{\lambda_{\min}}\left(b_1^{(n)}+\mu_n\right)^2\\
&=\frac{\lambda_{\min}}{\tau_n}+\frac{1}{\tau_n(1+r_n)^2(1+r_{n-1})^4(1+r_{n-1}+r_{n-1}r_n)^2}\cdot \Phi(r_n,r_{n-1})\geq \frac{\lambda_{\min}}{\tau_n}
\end{split}
\end{equation*}
with $\Phi(r_n,r_{n-1})\geq 0$ in Lemma \ref{Le:3.3}. Similarly, we have
\begin{equation*}
\begin{split}
p_n&=\widehat{b}_0^{(n)}-\frac{1}{p_{n-2}}\left(b^{(n)}_2\right)^2-\frac{1}{p_{n-1}}q^2_n 
\leq b_0^{(n)}-\frac{\tau_{n-2}}{\lambda_{\max}}\left(b^{(n)}_2\right)^2-\frac{\tau_{n-1}}{\lambda_{\max}}\left(b_1^{(n)}+\nu_n\right)^2\\
&=\frac{\lambda_{\max}}{\tau_n}+\frac{1}{\tau_n(1+r_n)^2(1+r_{n-1})^4(1+r_{n-1}+r_{n-1}r_n)^2}\cdot \phi(r_n,r_{n-1})\leq\frac{\lambda_{\max}}{\tau_n}
\end{split}
\end{equation*}
with $\phi(r_n,r_{n-1})\leq 0$ in Lemma \ref{Le:3.4}.
The proof is completed.
\end{proof}
\begin{remark}\label{re:3.1}
In fact, the upper ratio $r_s=1.405$  is  the root of the polynomial function  $\Phi(x,x)$ arising from Lemma   \ref{Le:3.3}.
\end{remark}

\section{The unique solvability and energy stability}\label{Se:solv}
In this section, we show the unique solvability and discrete energy stability.
Let $(\cdot,\cdot)$ and $\left\|\cdot\right\|$ be the usual inner product and  norm in the space $L^2(\Omega)$, respectively.
\subsection{The unique solvability}
First, we show the unique solvability of the BDF3 scheme  \eqref{1.7} via a discrete energy functional for the Allen-Cahn equation \eqref{1.1}.
\begin{theorem}\label{Theorem:4.1}
If the time-step size $\tau_n<\frac{1+2 r_n+r_{n-1}\left(1+4 r_n+3 r_n^2\right)}{(1+r_n)(1+r_{n-1}+r_nr_{n-1})}$, the variable-steps BDF3
 scheme  \eqref{1.7} is uniquely solvable.
\end{theorem}
\begin{proof}
For any fixed time-level indexes $n\geq1,$ we consider the following energy functional
\begin{equation*}
\begin{split}
G[z]:=\frac{b_0^{(n)}}{2}\left\|z-u^{n-1}\right\|^2\!+\!\left(b_1^{(n)}\nabla_\tau u^{n-1}
\!+\!b_2^{(n)}\nabla_\tau u^{n-2},z-u^{n-1}\right)\!+\!\frac{\varepsilon^2}{2}\left\|\nabla z\right\|^2\!+\!\frac{1}{4}\left\| z^2-1\right\|^2.
\end{split}
\end{equation*}
Under the time-step size condition $\tau_n<\frac{1+2 r_n+r_{n-1}\left(1+4 r_n+3 r_n^2\right)}{(1+r_n)(1+r_{n-1}+r_nr_{n-1})}$ or $b_0^{(n)}>1$, the functional $G$ is strictly convex.
In fact, for any $\lambda\in\mathbb{R}$ and  $\psi$, one has
\begin{equation*}
\begin{split}
\left.\frac{d^2G}{d\lambda^2}[z+\lambda\psi]\right|_{\lambda=0}
=b_0^{(n)}\left\|\psi\right\|^2+\varepsilon^2\left\|\nabla\psi\right\|^2+3\left\|z\psi\right\|^2-\left\|\psi\right\|^2
\geq\left(b_0^{(n)}-1\right)\left\|\psi\right\|^2>0.
\end{split}
\end{equation*}
Thus, the functional $G$ has a unique minimizer, denoted by $u^n$, if and only if it solves
\begin{equation*}
\begin{split}
0=\left.\frac{dG}{d\lambda}[z+\lambda\psi]\right|_{\lambda=0}
=\left(b_0^{(n)}\left(z-u^{n-1}\right)+b_1^{(n)}\nabla_\tau u^{n-1}
+b_2^{(n)}\nabla_\tau u^{n-2}-\varepsilon^2\Delta z+f(z),\psi\right).
\end{split}
\end{equation*}
This equation holds for any $\psi$ if and only if the unique minimizer $u^n$ solves
\begin{equation*}
\begin{split}
b_0^{(n)}\left(u^{n}-u^{n-1}\right)+b_1^{(n)}\nabla_\tau u^{n-1}
+b_2^{(n)}\nabla_\tau u^{n-2}-\varepsilon^2\Delta u^{n}+f(u^{n})=0,
\end{split}
\end{equation*}
which is just the BDF3 scheme \eqref{1.7}. The proof is completed.
\end{proof}

\subsection{The discrete energy dissipation law}
From \eqref{3.1} and Lemma \ref{Le:3.5},   for any real sequence $\{w_k\}_{k=1}^n$, it holds that
\begin{equation}\label{4.1}
\begin{split}
\sum_{k=1}^nw_k\sum_{j=1}^k b_{k-j}^{(k)}w_j\geq\gamma\sum_{k=1}^n\frac{w_k^2}{\tau_k},~~n\geq1.
\end{split}
\end{equation}

Let $E(u^n)$ be the discrete version of free energy functional \eqref{1.2}, given by
\begin{equation}\label{4.2}
\begin{split}
E(u^n)=\frac{\varepsilon^2}{2}\left\|\nabla u^n\right\|^2+\frac{1}{4}\left\| \left(u^n\right)^2-1\right\|^2,~~0\leq n\leq N.
\end{split}
\end{equation}
Next theorem shows that the variable steps BDF3 scheme \eqref{1.7} preserves an energy dissipation law at the discrete levels, which implies the energy stability.
\begin{theorem}\label{Theorem:4.2}
Let $r_n\leq 1.405$. If the time-step sizes are properly small such that
\begin{equation}\label{4.3}
\begin{split}
\tau_n\leq \min\left\{\frac{1+2 r_n+r_{n-1}\left(1+4 r_n+3 r_n^2\right)}{(1+r_n)(1+r_{n-1}+r_nr_{n-1})},2\gamma\right\},~~n\geq1.
\end{split}
\end{equation}
Then the variable-steps BDF3  scheme  \eqref{1.7} preserves the following energy dissipation law
\begin{equation}\label{4.4}
\begin{split}
E(u^n)\leq E(u^0),~~n\geq1.
\end{split}
\end{equation}
\end{theorem}
\begin{proof}
The first condition of \eqref{4.3} ensures the unique solvability. We will establish the energy dissipation law under the second condition of \eqref{4.3}.
Making the inner product of \eqref{1.7} by $\nabla_{\tau}u^k$, we obtain
\begin{equation}\label{4.5}
\begin{split}
\left(D_3u^k,\nabla_{\tau}u^k\right)-\varepsilon^2\left(\Delta u^k,\nabla_{\tau}u^k\right)+\left(f(u^k),\nabla_{\tau}u^k\right)=0.
\end{split}
\end{equation}
With the help of the inequality $2a(a-b)\geq a^2-b^2$, the second term in \eqref{4.5} reads
\begin{equation*}
\begin{split}
-\varepsilon^2\left(\Delta u^k,\nabla_{\tau}u^k\right)=\varepsilon^2\left(\nabla u^k,\nabla u^k-\nabla u^{k-1}\right)
\geq \frac{\varepsilon^2}{2}\left\|\nabla u^k\right\|^2
-\frac{\varepsilon^2}{2}\left\|\nabla u^{k-1}\right\|^2.
\end{split}
\end{equation*}
It is easy to check the following identity
\begin{equation*}
\begin{split}
4\left(a^3-a\right)\left(a-b\right)=\left(a^2-1\right)^2-\left(b^2-1\right)^2
-2\left(a-b\right)^2+2a^2\left(a-b\right)^2+\left(a^2-b^2\right)^2.
\end{split}
\end{equation*}
Then the third term in \eqref{4.5} can be bounded by
\begin{equation*}
\begin{split}
\left(f(u^k),\nabla_{\tau}u^k\right)\geq\frac{1}{4}\left\| \left(u^k\right)^2-1\right\|^2
-\frac{1}{4}\left\| \left(u^{k-1}\right)^2-1\right\|^2
-\frac{1}{2}\left\|u^k-u^{k-1}\right\|^2.
\end{split}
\end{equation*}
From \eqref{4.5} and the above inequalities, it yields
\begin{equation*}
\begin{split}
\left(D_3u^k,\nabla_{\tau}u^k\right)+E(u^k)-E(u^{k-1})
-\frac{1}{2}\left\|u^k-u^{k-1}\right\|^2\leq0,~~k\geq1.
\end{split}
\end{equation*}
Summing the above inequality from $k=1$ to $n$, we have
\begin{equation*}
\begin{split}
\sum_{k=1}^n\left(D_3u^k,\nabla_{\tau}u^k\right)+E(u^n)-E(u^0)-\frac{1}{2}\sum_{k=1}^n\left\|u^k-u^{k-1}\right\|^2\leq0.
\end{split}
\end{equation*}
According to  \eqref{1.8} and \eqref{4.1}, we obtain
\begin{equation*}
\begin{split}
\sum_{k=1}^n\left(D_3u^k,\nabla_{\tau}u^k\right)=\sum_{k=1}^n\left(\sum_{j=1}^k b_{k-j}^{(k)}\nabla_{\tau}u^j,\nabla_{\tau}u^k\right)
\geq\gamma\sum_{k=1}^n\frac{\left\|u^k-u^{k-1}\right\|^2}{\tau_k},~~n\geq1.
\end{split}
\end{equation*}
Hence, it implies
\begin{equation*}
\begin{split}
\sum_{k=1}^n\left(\frac{\gamma}{\tau_k}-\frac{1}{2}\right)\left\|u^k-u^{k-1}\right\|^2+E(u^n)-E(u^0)\leq0.
\end{split}
\end{equation*}
The second condition of \eqref{4.3} gives the desired result \eqref{4.4}.
\end{proof}

\begin{lemma}\label{Lemma:4.1}
Let $r_n\leq1.405$.  If the step sizes $\tau_n$ fulfill \eqref{4.3}, the solution of the variable steps BDF3 scheme \eqref{1.7} is bounded in the sense that
\begin{equation*}
\begin{split}
\left\|u^n\right\|+\left\|\nabla u^n\right\|\leq c_1:=\sqrt{4\varepsilon^{-2}E(u^0)+\left(2+\varepsilon^2\right)\left|\Omega\right|},~~n\geq1,
\end{split}
\end{equation*}
where $c_1$ is dependent on the domain $\Omega$ and the starting value $u^0$, but
independent of the time $t_n$, the time-step sizes $\tau_n$ and the time-step ratios $r_n$.
\end{lemma}
\begin{proof}
From the discrete energy dissipation law \eqref{4.4} and the definition \eqref{4.2}, it yields
\begin{equation*}
\begin{split}
4E(u^0)&\geq 4E(u^n)=2\varepsilon^2\left\|\nabla u^n\right\|^2+\left\|\left(u^n\right)^2-1\right\|^2\\
&=2\varepsilon^2\left\|\nabla u^n\right\|^2
+\left\|\left(u^n\right)^2-1-\varepsilon^2\right\|^2
+2\varepsilon^2\left\|u^n\right\|^2
-\varepsilon^2\left(2+\varepsilon^2\right)\left|\Omega\right|\\
&\geq2\varepsilon^2\left\|\nabla u^n\right\|^2+2\varepsilon^2\left\|u^n\right\|^2-\varepsilon^2\left(2+\varepsilon^2\right)\left|\Omega\right|.
\end{split}
\end{equation*}
Thus, we obtain
\begin{equation*}
\begin{split}
\left(\left\|u^n\right\|+\left\|\nabla u^n\right\|\right)^2
\leq 2\left\|u^n\right\|^2+2\left\|\nabla u^n\right\|^2
\leq 4\varepsilon^{-2}E(u^0)+\left(2+\varepsilon^2\right)\left|\Omega\right|.
\end{split}
\end{equation*}
The proof is completed.
\end{proof}

\section{Stability and convergence analysis}\label{Se:stab}
In this section, we show the $L^2$ norm unconditional stability and convergence  of the variable-step BDF3 scheme  \eqref{1.7} for the Allen-Cahn equation.

Denote  $\langle\cdot,\cdot\rangle$ the  classical Euclidean scalar product
\begin{equation*}
\begin{split}
\langle \mu,\nu\rangle=\nu^{T}\mu=\sum_{k=1}^n\mu^k\nu^k,~~~
\left|\mu\right|=\langle \mu,\mu\rangle^{1/2}
\end{split}
\end{equation*}
with $\mu=(\mu^1,\mu^2,\cdots,\mu^n)^T$ and $\nu=(\nu^1,\nu^2,\cdots,\nu^n)^T$.
From  \cite[pp.\,23-24]{Quarteroni:07}, we know that the  spectral norm of the matrix $A \in \mathbb{R}^{n\times n}$ satisfies
\begin{equation}\label{5.1}
\left|A\mu\right|\leq \left|A\right|\left|\mu\right|~~~~{\rm with}~~~~\left|A\right|=\sqrt{\rho\left(A^TA\right)}.
\end{equation}
Here the spectral radius $\rho(A)$   is denoted by the maximum module of the eigenvalues of $A$.

\begin{definition}\label{Definition:5.1}
Let $A$ and $B$ be two real $n\times n$ matrices. Then,  $A> B$ ($\geq B$) if $A-B$ is positive definite (positive semi-definite).
\end{definition}

Let $I$  be the  $n\times n$   identity matrix and $\Lambda={\rm diag}\left(\tau_1,\tau_2, \ldots, \tau_n  \right)$ in \eqref{2.1}.
Then we have the following  results.

\begin{lemma}\label{Lemma:5.1}
Let $B>c\Lambda^{-1}$, $c>0$.
Then $\mathscr{A}:=\Lambda^{1/2}B\Lambda^{1/2}>cI$.
\end{lemma}
\begin{proof}
Taking  $x=\Lambda^{1/2}y$ Bith $x\neq 0$, it yields
\begin{equation*}
\begin{split}
0<x^T\left(B- c\Lambda^{-1}\right)x=y^T\Lambda^{1/2}\left(B- c\Lambda^{-1}\right)\Lambda^{1/2}y=y^T\left( \Lambda^{1/2}B\Lambda^{1/2}-cI \right)y.
\end{split}
\end{equation*}
The proof is completed.
\end{proof}
\begin{lemma}[Spectral norm inequality]\label{Lemma:5.2}
Let $A>cI$, $c>0$.
Then  the  spectral norm $\left|A^{-1}\right|<c^{-1}$.
\end{lemma}
\begin{proof}
Since
\begin{equation*}
\begin{split}
x^T\left(A- c I\right)x>0, ~~~x^T\left(A^T- c I\right)x>0 \quad\forall x\neq0.
\end{split}
\end{equation*}
Using the  the  classical Euclidean scalar product, it yields
\begin{equation*}
\begin{split}
0<\left|\left(A- c I\right)x\right|^2
&=x^T\left(A^T- c I\right)\left(A- c I\right)x
=x^T\left(A^TA- c A^T- c A+c^2I\right)x\\
&=x^T\left(A^TA- c^2  I- c A^T- c A+2c^2I\right)x \quad\forall x\neq0.
\end{split}
\end{equation*}
According to the above inequalities, we have
\begin{equation}\label{5.2}
\begin{split}
x^T\left(A^TA- c^2  I\right)x
&> c x^T\left(A^T+A-2cI\right)x>0 \quad\forall x\neq0,
\end{split}
\end{equation}
which implies that   the matrix $A^TA$ is symmetric positive definite.
 Let $\{\mu_i\}_{i=1}^n$ be an orthonormal set of eigenvectors of $A^TA$, i.e., $A^TA\mu_i=\lambda_i\mu_i$
with $0<\lambda_1\leq\lambda_2\leq\cdots\leq\lambda_n$.
Thus, we obtain
\begin{equation*}
\begin{split}
x^TA^TAx=\sum_{i=1}^nc_i^2\lambda_i,~~ x^Tx=\sum_{i=1}^nc_i^2  \quad\forall x=\sum_{i=1}^nc_i\mu_i.
\end{split}
\end{equation*}
From \eqref{5.2} and the above equations, there exists
\begin{equation*}
\begin{split}
x^T\left(A^TA- c^2  I\right)x=\sum_{i=1}^nc_i^2\left(\lambda_i- c^2  \right)>0 \quad\forall x\neq0,
\end{split}
\end{equation*}
which leads to $\lambda_1> c^2.$
From \eqref{5.1}, one has
\begin{equation*}
\begin{split}
\rho\left(\left(A^TA\right)^{-1}\right)=\lambda_1^{-1}< 1/c^2~~{\rm and}~~\left|A^{-1}\right|<c^{-1}.
\end{split}
\end{equation*}
The proof is completed.
\end{proof}
\begin{lemma}\label{Lemma:5.3}
If the BDF3 discrete convolution kernels $b_{n-k}^{(n)}$ in \eqref{1.8} are positive definite, the DOC kernels $d_{n-k}^{(n)}$ in \eqref{1.9} are also positive definite. For any real sequence $\{\mu^k\}_{k=1}^n$, it holds that
\begin{equation*}
\begin{split}
\sum_{k=1}^n\mu^k\sum_{j=1}^kd_{k-j}^{(k)}\mu^j\geq0,~~n\geq1.
\end{split}
\end{equation*}
\end{lemma}
\begin{proof}
Let $\mu=(\mu^1,\mu^2,\cdots,\mu^n)^T\in \mathbb{R}^{n}$. We can check
\begin{equation*}
\begin{split}
\sum^n_{k=1}\mu^k\sum^k_{j=1}d^{(k)}_{k-j}\mu^j=\mu^TD \mu,~~n \geq 1
\end{split}
\end{equation*}
with the matrix $D$  in \eqref{1.11}.

According to Lemma \ref{Le:3.5}, we know that the matrix $B$ is positive definite.
Let $\forall \mu\in  \mathbb{R}^{n}$, $\mu\neq 0$, it yields  $\nu=B\mu\neq0$. Then we have
$$\nu^TD\nu=\nu^TB^{-1}\nu=\mu^TB^TB^{-1}B\mu=\mu^TB^T\mu>0.$$
The proof is completed.
\end{proof}

A discrete Gr\"{o}nwall's inequality is needed in the following analysis.
\begin{lemma}\cite{LZ}\label{Lemma:5.4}
Let $\lambda\geq0$ and the sequences $\left\{\xi_k\right\}_{k=0}^N$ and $\left\{V_k\right\}_{k=1}^N$ be nonnegative. If
\begin{equation*}
\begin{split}
V_n\leq\lambda\sum_{j=1}^{n-1}\tau_jV_j+\sum_{j=0}^{n}\xi_j,~~1\leq n\leq N,
\end{split}
\end{equation*}
then it holds that
\begin{equation*}
\begin{split}
V_n\leq\exp\left(\lambda t_{n-1}\right)\sum_{j=0}^{n}\xi_j,~~1\leq n\leq N.
\end{split}
\end{equation*}
\end{lemma}

\subsection{Stability  analysis}\label{Se:staana}
First, we show the $L^2$ norm stability analysis of the variable steps BDF3 scheme \eqref{1.7} for the Allen-Cahn model \eqref{1.1}.
\begin{theorem}\label{Theorem:5.1}
Let  BDF3 kernels $b_{n-k}^{(n)}$  be defined in \eqref{1.8}  with $B>\gamma\Lambda^{-1}$ in \eqref{3.1}.
Then the discrete solution $u^n$ of  BDF3 scheme \eqref{1.7} is unconditionally stable in the $L^2$ norm
\begin{equation*}
\begin{split}
\left\|\epsilon^{n}\right\|\leq 2\exp\left(4\gamma^{-1}\widetilde{c}t_{n-1}\right)\left\|\epsilon^{0}\right\|,~~n\geq1.
\end{split}
\end{equation*}
\end{theorem}
\begin{proof}
Let $\epsilon^{n}$ be the solution perturbation $\epsilon^{n}=\widetilde{u}^n-u^n$ for $0\leq n\leq N$. The perturbed equation is obtained
\begin{equation}\label{5.3}
\begin{split}
D_3\epsilon^j-\varepsilon^2\Delta \epsilon^j=f(u^j)-f(\widetilde{u}^j)=\widetilde{f}_u^j\epsilon^j, ~~j \geq 1
\end{split}
\end{equation}
with
\begin{equation*}
\begin{split}
\widetilde{f}_u^j=1-\left(u^j\right)^2-u^j\widetilde{u}^j-\left(\widetilde{u}^j\right)^2.
\end{split}
\end{equation*}
Note that the solution estimates in Lemma \ref{Lemma:4.1} and $H^1\subseteq L^\infty$, we have
\begin{equation}\label{5.4}
\begin{split}
\left\|\widetilde{f}_u^j\right\|_{L^\infty}
&\leq\left|\Omega\right|+\left\|u^j\right\|_{L^\infty}^2
+\left\|u^j\right\|_{L^\infty}\left\|\widetilde{u}^j\right\|_{L^\infty}
+\left\|\widetilde{u}^j\right\|_{L^\infty}^2\\
&\leq\left|\Omega\right|+c_\Omega^2\left\|u^j\right\|_{H^1}^2+c_\Omega^2
\left\|u^j\right\|_{H^1}\left\|\widetilde{u}^j\right\|_{H^1}+c_\Omega^2\left\|\widetilde{u}^j\right\|_{H^1}^2\\
&\leq\left|\Omega\right|+c_\Omega^2c_1^2+c_\Omega^2c_1\widetilde{c_1}+c_\Omega^2\widetilde{c_1}^2:=\widetilde{c},~~j\geq1,
\end{split}
\end{equation}
where $\left\|\widetilde{u}^j\right\|_{H^1}\leq \widetilde{c_1}$ is similar to Lemma \ref{Lemma:4.1}.
Multiplying both sides of \eqref{5.3} by the DOC kernels $d_{k-j}^{(k)}$, and summing $j$ from $1$ to $k$, we derive
\begin{equation*}
\begin{split}
\sum_{j=1}^kd_{k-j}^{(k)}D_3\epsilon^j-\varepsilon^2\sum_{j=1}^kd_{k-j}^{(k)}\Delta \epsilon^j=\sum_{j=1}^kd_{k-j}^{(k)}\widetilde{f}_u^j\epsilon^j,~~k \geq 1.
\end{split}
\end{equation*}
According to \eqref{1.8} and  \eqref{1.10}, it yields
 \begin{equation}\label{5.5}
 \begin{split}
\sum_{j=1}^kd_{k-j}^{(k)}D_3\epsilon^j
=\sum_{j=1}^kd_{k-j}^{(k)}\sum_{l=1}^{j}b_{j-l}^{(j)}\nabla_\tau \epsilon^l
=\sum_{l=1}^{k}\nabla_\tau \epsilon^l\sum_{j=l}^kd_{k-j}^{(k)}b_{j-l}^{(j)}=\nabla_\tau \epsilon^k,~~k \geq 1.
\end{split}
\end{equation}
Hence, we have
\begin{equation*}
\begin{split}
\nabla_\tau \epsilon^k-\varepsilon^2\sum_{j=1}^kd_{k-j}^{(k)}\Delta \epsilon^j
=\sum_{j=1}^kd_{k-j}^{(k)}\widetilde{f}_u^j\epsilon^j,~~k \geq 1.
\end{split}
\end{equation*}
Making the inner product of the above equality with $\epsilon^k$ and summing the dervied equality from $k=1$ to $n$, one obtains
\begin{equation*}
\begin{split}
\sum_{k=1}^n\left(\nabla_\tau \epsilon^k,\epsilon^k\right)+\varepsilon^2\sum_{k=1}^{n}\sum_{j=1}^kd_{k-j}^{(k)}\left(\nabla \epsilon^j,\nabla \epsilon^k\right)=\sum_{k=1}^{n}\sum_{j=1}^kd_{k-j}^{(k)}\left(\widetilde{f}_u^j\epsilon^j,\epsilon^k\right),~~n \geq 1.
\end{split}
\end{equation*}
For the first term on the left hand, we have
\begin{equation*}
\begin{split}
\sum_{k=1}^n\left(\nabla_\tau \epsilon^k,\epsilon^k\right)
\geq\frac{1}{2}\sum_{k=1}^n\left(\left\|\epsilon^k\right\|^2-\left\|\epsilon^{k-1}\right\|^2\right)
=\frac{1}{2}\left(\left\|\epsilon^n\right\|^2-\left\|\epsilon^0\right\|^2\right),
\end{split}
\end{equation*}
where the inequality $2a(a-b)\geq a^2-b^2$ has been used.

For the second term on the left hand, using  Lemma \ref{Lemma:5.3}, we obtain
\begin{equation*}
\begin{split}
\varepsilon^2\sum_{k=1}^{n}\sum_{j=1}^kd_{k-j}^{(k)}\left(\nabla \epsilon^j,\nabla \epsilon^k\right)\geq0.
\end{split}
\end{equation*}
From the above estimates, \eqref{1.11}, Lemma \ref{Lemma:5.1}, discrete Cauchy-Schwarz inequality and \eqref{3.1}, it yields
\begin{equation*}
\begin{split}
\left\|\epsilon^n\right\|^2-\left\|\epsilon^0\right\|^2
&\leq2\sum_{k=1}^{n}\sum_{j=1}^kd_{k-j}^{(k)}\left(\widetilde{f}_u^j\epsilon^j,\epsilon^k\right)
=2\int_\Omega \mathcal{E}^TDF_\epsilon dx
=2\int_\Omega \left\langle DF_\epsilon,\mathcal{E}\right\rangle dx\\
&=2\int_\Omega \left\langle B^{-1}F_\epsilon,\mathcal{E}\right\rangle dx
=2\int_\Omega \left\langle \mathscr{A}^{-1}\Lambda^{1/2} F_\epsilon,\Lambda^{1/2}\mathcal{E}\right\rangle dx\\
&\leq2\int_\Omega \left|\mathscr{A}^{-1}\right| \left|\Lambda^{1/2}F_\epsilon\right|\left|\Lambda^{1/2}\mathcal{E}\right|dx
\end{split}
\end{equation*}
with
\begin{equation*}
\begin{split}
\mathcal{E}=\left(\epsilon^1,\epsilon^2,\cdots,\epsilon^n\right)^{T}, ~~F_\epsilon=\left(\widetilde{f}_u^1\epsilon^1,\widetilde{f}_u^2\epsilon^2,\cdots,\widetilde{f}_u^n\epsilon^n\right)^{T}.
\end{split}
\end{equation*}
According to Lemmas \ref{Lemma:5.1}, \ref{Lemma:5.2}, Cauchy-Schwarz inequality and \eqref{5.4}, we get
\begin{equation*}
\begin{split}
\left\|\epsilon^n\right\|^2-\left\|\epsilon^0\right\|^2
&\leq 2\gamma^{-1}\int_{\Omega}\left|\Lambda^{1/2}F_\epsilon\right|\left|\Lambda^{1/2}\mathcal{E}\right|dx
\leq 2\gamma^{-1}\sqrt{\int_{\Omega}\left|\Lambda^{1/2}F_\epsilon\right|^2dx}
\sqrt{\int_{\Omega}\left|\Lambda^{1/2}\mathcal{E}\right|^2dx}\\
&=2\gamma^{-1}\sqrt{\int_{\Omega}\sum_{k=1}^{n}\tau_k\left(\widetilde{f}_u^k\epsilon^k\right)^2dx}
\sqrt{\int_{\Omega}\sum_{k=1}^{n}\tau_k\left(\epsilon^k\right)^2dx}\\
&\leq 2\gamma^{-1}\widetilde{c}\int_{\Omega}\sum_{k=1}^{n}\tau_k\left(\epsilon^k\right)^2dx
= 2\gamma^{-1}\widetilde{c}\sum_{k=1}^{n}\tau_k\left\|\epsilon^k\right\|^2,~~n \geq 1.
\end{split}
\end{equation*}
Choosing some integer $n_1\left(0\leq n_1\leq n\right)$ such that $\left\|\epsilon^{n_1}\right\|=\max_{0\leq k\leq n}\left\|\epsilon^k\right\|$. Taking $n:=n_1$ in the above inequality, we get
\begin{equation*}
\begin{split}
\left\|\epsilon^{n}\right\|
\leq\left\|\epsilon^0\right\|+2\gamma^{-1}\widetilde{c}\sum_{k=1}^{n}\tau_k\left\|\epsilon^k\right\|,~~n \geq 1.
\end{split}
\end{equation*}
Using the discrete Gr\"{o}nwall's inequality in Lemma \ref{Lemma:5.4} and for sufficiently small step-sizes $\tau_n$
(namely., $2\gamma^{-1}\widetilde{c}\tau_n<\frac{1}{2}$), we get
\begin{equation*}
\begin{split}
\left\|\epsilon^{n}\right\|\leq 2\exp\left(4\gamma^{-1}\widetilde{c}t_{n-1}\right)\left\|\epsilon^{0}\right\|,~~n\geq1.
\end{split}
\end{equation*}
The proof is completed.
\end{proof}

\subsection{Convergence  analysis}\label{Se:conv}
We are now at the stage to show the $L^2$ norm convergence analysis. We first consider the consistency error of  the variable steps BDF3 scheme \eqref{1.7}.
\begin{lemma}\label{Lemma:5.5}
For the consistency error $\eta^j:=D_3u(t_j)-\partial_tu(t_j)$ for $j\geq 1$, it holds that
\begin{equation*}
\begin{split}
\left\|\eta^1\right\|&\leq \int_{0}^{t_1}\left\|\partial_{tt}u\right\|dt,\quad
\left\|\eta^2\right\|\leq 2\tau_2\int_{t_1}^{t_2}\left\|\partial_{ttt}u\right\|
+\frac{1}{2}\tau_1\int_{0}^{t_1}\left\|\partial_{ttt}u\right\|,\\
\left\|\eta^j\right\|&\leq C\left(\tau_{j}^2\int^{t_j}_{t_{j-1}}\left\|\partial_{tttt}u\right\|dt
+\tau_{j-1}^2\int^{t_{j-1}}_{t_{j-2}}\left\|\partial_{tttt}u\right\|dt
+\tau_{j-2}^2\int^{t_{j-2}}_{t_{j-3}}\left\|\partial_{tttt}u\right\|dt\right),~~ j\geq 3.
\end{split}
\end{equation*}
\end{lemma}
\begin{proof}
For simplicity, denote
\begin{equation*}
\begin{split}
G_{3}^j=\int_{t_{j-1}}^{t_j}\left\|\partial_{ttt}u\right\|dt \quad{\rm  and}\quad
G_{4}^j=\int_{t_{j-1}}^{t_j}\left\|\partial_{tttt}u\right\|dt,~~ j\geq 1.
\end{split}
\end{equation*}
For the cases of $j=1$ and $j=2$, according to Lemma 4.1 in \cite{CYZ:21}, we have
\begin{equation*}
\begin{split}
\left\|\eta^1\right\|&\leq\int_{0}^{t_1}\left\|\partial_{tt}u\right\|dt,\\
\left\|\eta^2\right\|&\leq \frac{1+2 r_2}{1+r_2}\tau_2G_{3}^2+\frac{r_2}{2\left(1+r_2\right)}\tau_1G_{3}^{1}
\leq2\tau_2\int_{t_1}^{t_2}\left\|\partial_{ttt}u\right\|
+\frac{1}{2}\tau_1\int_{0}^{t_1}\left\|\partial_{ttt}u\right\|.
\end{split}
\end{equation*}
For the case of $j\geq3$, by using the Taylor's expansion formula, it yields
\begin{equation*}
\begin{split}
\eta^j&=\frac{b_1^{(j)}-b_0^{(j)}}{6}\int^{t_j}_{t_{j-1}}\left(t-t_{j-1}\right)^3\partial_{tttt}udt
+\frac{b_2^{(j)}-b_1^{(j)}}{6}\int^{t_j}_{t_{j-2}}\left(t-t_{j-2}\right)^3\partial_{tttt}udt\\
&\quad-\frac{b_2^{(j)}}{6}\int^{t_j}_{t_{j-3}}\left(t-t_{j-3}\right)^3\partial_{tttt}udt.
\end{split}
\end{equation*}
According to \eqref{1.5}, the consistency error is bounded by
\begin{equation*}
\begin{split}
\left\|\eta^j\right\|&\leq C\left(b_0^{(j)}\tau_j^3G_{4}^j-b_1^{(j)}\tau_{j-1}^3G_{4}^{j-1}
+b_2^{(j)}\tau_{j-2}^3G_{4}^{j-2}\right)\\
&\leq C\left(\tau_{j}^2\int^{t_j}_{t_{j-1}}\left\|\partial_{tttt}u\right\|dt
+\tau_{j-1}^2\int^{t_{j-1}}_{t_{j-2}}\left\|\partial_{tttt}u\right\|dt
+\tau_{j-2}^2\int^{t_{j-2}}_{t_{j-3}}\left\|\partial_{tttt}u\right\|dt\right).
\end{split}
\end{equation*}
The proof is completed.
\end{proof}
\begin{theorem}\label{Theorem:5.2}
Let $u(t_n)$ and $u^n$ be the solution of  \eqref{1.1} and the BDF3 scheme \eqref{1.7}, respectively. Then   the following error estimate holds for $1\leq n\leq N$
\begin{equation*}
\begin{split}
\|u(t_n)-u^n\|
&\leq C\left(\tau_1^{1/2}\int_{0}^{t_1}\left\|\partial_{tt}u\right\|dt+ \tau_2^{3/2}\int_{t_1}^{t_2}\left\|\partial_{ttt}u\right\|
+\tau_1\tau_2^{1/2}\int_{0}^{t_1}\left\|\partial_{ttt}u\right\|\right.\\
&\left.\!+\!\sqrt{\sum_{k=3}^{n}\!\tau_k\left(\!\tau_{k}^2\!\int^{t_k}_{t_{k-1}}\!\left\|\partial_{tttt}u\right\|\!dt
\!+\!\tau_{k-1}^2\!\int^{t_{k-1}}_{t_{k-2}}\!\left\|\partial_{tttt}u\right\|\!dt
\!+\!\tau_{k-2}^2\!\int^{t_{k-2}}_{t_{k-3}}\!\left\|\partial_{tttt}u\right\|\!dt\!\right)^2}\right).
\end{split}
\end{equation*}
\end{theorem}
\begin{proof}
Let $e^n:=u(t_n)-u^n$ be the error function  with $e^0=0$. From \eqref{1.1} and \eqref{1.7},
we have the following error equation
\begin{equation}\label{5.6}
\begin{split}
D_3e^j-\varepsilon^2\Delta e^j=f(u^j)-f(u(t_j))+\eta^j=f_u^je^j+\eta^j, ~~j \geq 1
\end{split}
\end{equation}
with
\begin{equation*}
\begin{split}
f_u^j=1-\left(u^j\right)^2-u^ju(t_j)-\left(u(t_j)\right)^2,~~~\eta^j=D_3u(t_j)-u'(t_j).
\end{split}
\end{equation*}
The energy dissipation law \eqref{1.3} of the Allen-Cahn equation \eqref{1.1} shows that $E(u(t_n))\leq E(u(t_0))$. From the formulation \eqref{1.2}, it is easy to check that $\left\|u(t_n)\right\|_{H^1}$ can be bounded by a time-independent constant $c_2$.
Note that the solution estimates in Lemma \ref{Lemma:4.1} and $H^1\subseteq L^\infty$, we have
\begin{equation}\label{5.7}
\begin{split}
\left\|f_u^j\right\|_{L^\infty}&\leq\left|\Omega\right|+\left\|u^j\right\|_{L^\infty}^2+
\left\|u^j\right\|_{L^\infty}\left\|u(t_j)\right\|_{L^\infty}+\left\|u(t_j)\right\|_{L^\infty}^2\\
&\leq\left|\Omega\right|+c_\Omega^2\left\|u^j\right\|_{H^1}^2+c_\Omega^2
\left\|u^j\right\|_{H^1}\left\|u(t_j)\right\|_{H^1}+c_\Omega^2\left\|u(t_j)\right\|_{H^1}^2\\
&\leq\left|\Omega\right|+c_\Omega^2c_1^2+c_\Omega^2c_1c_2+c_\Omega^2c_2^2:=c_3,~~j\geq1.
\end{split}
\end{equation}
Multiplying both sides of \eqref{5.6} by the DOC kernels $d_{k-j}^{(k)}$, and summing $j$ from $1$ to $k$, we derive by applying the equality \eqref{5.5}
\begin{equation*}
\begin{split}
\nabla_\tau e^k-\varepsilon^2\sum_{j=1}^kd_{k-j}^{(k)}\Delta e^j=\sum_{j=1}^kd_{k-j}^{(k)}f_u^je^j+\sum_{j=1}^kd_{k-j}^{(k)}\eta^j,~~k \geq 1.
\end{split}
\end{equation*}
Making the inner product of the above equality with $e^k$ and summing the resulting equality from $k=1$ to $n$, there exists
\begin{equation*}
\begin{split}
\sum_{k=1}^n\left(\nabla_\tau e^k,e^k\right)\!+\!\varepsilon^2\sum_{k=1}^{n}\sum_{j=1}^kd_{k-j}^{(k)}\left(\nabla e^j,\nabla e^k\right)\!=\!\sum_{k=1}^{n}\sum_{j=1}^kd_{k-j}^{(k)}\left(f_u^je^j,e^k\right)
\!+\!\sum_{k=1}^{n}\sum_{j=1}^kd_{k-j}^{(k)}\left(\eta^j,e^k\right).
\end{split}
\end{equation*}
For the first term on the left hand, we have
\begin{equation*}
\begin{split}
\sum_{k=1}^n\left(\nabla_\tau e^k,e^k\right)
\geq\frac{1}{2}\sum_{k=1}^n\left(\left\|e^k\right\|^2-\left\|e^{k-1}\right\|^2\right)
=\frac{1}{2}\left(\left\|e^n\right\|^2-\left\|e^0\right\|^2\right),
\end{split}
\end{equation*}
where the inequality $2a(a-b)\geq a^2-b^2$ has been used.

For the second term on the left hand, using  Lemma \ref{Lemma:5.3}, we obtain
\begin{equation*}
\begin{split}
\sum_{k=1}^{n}\sum_{j=1}^kd_{k-j}^{(k)}\left(\nabla e^j,\nabla e^k\right)\geq0.
\end{split}
\end{equation*}
From the above estimates, \eqref{1.11}, Lemma \ref{Lemma:5.1}, discrete Cauchy-Schwarz inequality and \eqref{3.1}, it yields
\begin{equation*}
\begin{split}
\left\|e^n\right\|^2-\left\|e^0\right\|^2
&\leq2\sum_{k=1}^{n}\sum_{j=1}^kd_{k-j}^{(k)}\left(f_u^je^j,e^k\right)
+2\sum_{k=1}^{n}\sum_{j=1}^kd_{k-j}^{(k)}\left(\eta^j,e^k\right)\\
&=2\int_\Omega E^TDF_edx+2\int_\Omega E^TD\Upsilon dx
=2\int_\Omega \left\langle DF_e,E\right\rangle dx
+2\int_\Omega \left\langle D\Upsilon,E\right\rangle dx\\
&=2\int_\Omega \left\langle B^{-1}F_e,E\right\rangle dx
+2\int_\Omega \left\langle B^{-1}\Upsilon,E\right\rangle dx\\
&=2\int_\Omega \left\langle\mathscr{A}^{-1}\Lambda^{1/2}F_e,\Lambda^{1/2}E\right\rangle dx
+2\int_\Omega \left\langle \mathscr{A}^{-1}\Lambda^{1/2} \Upsilon,\Lambda^{1/2}E\right\rangle dx\\
&\leq2\int_\Omega \left|\mathscr{A}^{-1}\right| \left|\Lambda^{1/2}F_e\right|\left|\Lambda^{1/2}E\right|dx
+2\int_\Omega \left|\mathscr{A}^{-1}\right| \left|\Lambda^{1/2}\Upsilon\right|\left|\Lambda^{1/2}E\right|dx
\end{split}
\end{equation*}
with
\begin{equation*}
\begin{split}
E=\left(e^1,e^2,\cdots,e^n\right)^{T}, ~~F_e=\left(f_u^1e^1,f_u^2e^2,\cdots,f_u^ne^n\right)^{T},
~~\Upsilon=\left(\eta^1,\eta^2,\cdots,\eta^n\right)^{T}.
\end{split}
\end{equation*}
According to Lemmas \ref{Lemma:5.1}, \ref{Lemma:5.2}, Cauchy-Schwarz inequality and \eqref{5.7}, we get
\begin{equation*}
\begin{split}
\left\|e^n\right\|^2-\left\|e^0\right\|^2
&\leq 2\gamma^{-1}\int_{\Omega}\left|\Lambda^{1/2}F_e\right|\left|\Lambda^{1/2}E\right|dx
+2\gamma^{-1}\int_{\Omega}\left|\Lambda^{1/2}\Upsilon\right|\left|\Lambda^{1/2}E\right|dx\\
&\leq 2\gamma^{-1}\left(\sqrt{\int_{\Omega}\left|\Lambda^{1/2}F_e\right|^2dx}
\!+\!\sqrt{\int_{\Omega}\left|\Lambda^{1/2}\Upsilon\right|^2dx}\right)
\sqrt{\int_{\Omega}\left|\Lambda^{1/2}E\right|^2dx}\\
&=2\gamma^{-1}\left(\sqrt{\int_{\Omega}\sum_{k=1}^{n}\tau_k\left(f_u^ke^k\right)^2dx}
\!+\!\sqrt{\int_{\Omega}\sum_{k=1}^{n}\tau_k\left(\eta^k\right)^2dx}\right)
\sqrt{\int_{\Omega}\sum_{k=1}^{n}\tau_k\left(e^k\right)^2dx}\\
&\leq 2\gamma^{-1}c_3\sum_{k=1}^{n}\tau_k\left\|e^k\right\|^2
+2\gamma^{-1}\sqrt{\sum_{k=1}^{n}\tau_k\left\|\eta^k\right\|^2}
\sqrt{\sum_{k=1}^{n}\tau_k\left\|e^k\right\|^2},~~n \geq 1.
\end{split}
\end{equation*}
Taking some integer $n_2\left(0\leq n_2\leq n\right)$ such that $\left\|e^{n_2}\right\|=\max_{0\leq k\leq n}\left\|e^k\right\|$. Setting $n:=n_2$ in the above inequality, we get
\begin{equation*}
\begin{split}
\left\|e^{n}\right\|\leq\left\|e^0\right\|+2\gamma^{-1}c_3\sum_{k=1}^{n}\tau_k\left\|e^k\right\|
+2\gamma^{-1}\sqrt{\sum_{k=1}^{n}\tau_k\left\|\eta^k\right\|^2},~~n \geq 1.
\end{split}
\end{equation*}
Using the discrete Gr\"{o}nwall's inequality in Lemma \ref{Lemma:5.4} and for sufficiently small step-sizes $\tau_n$
(namely., $2\gamma^{-1}c_3\tau_n<\frac{1}{2}$), we get
\begin{equation*}
\begin{split}
\left\|e^{n}\right\|\leq 4\gamma^{-1}\exp\left(4\gamma^{-1}c_3t_{n-1}\right)\sqrt{\sum_{k=1}^{n}\tau_k\left\|\eta^k\right\|^2},~~n \geq 1.
\end{split}
\end{equation*}
The desired result follows by Lemma \ref{Lemma:5.5}.
\end{proof}

\section{Numerical example}
In this section, we  provide some details on the numerical implementations and present several numerical examples to confirm our theoretical statements. For the variable steps
BDF3 scheme \eqref{1.7} for the Allen-Cahn equation, we perform a simple Newton-type iteration  at each time level with a tolerance $10^{-10}$. We always choose the solution at the previous level as the initial value of Newton iteration.
In space, we discretize by the spectral collocation method  with the Cheby\-shev--Gauss--Lobatto points.
We numerically verify the  theoretical results including convergence orders in the discrete $L^2$-norm.
In order to investigate the temporal convergence rate at the final time $T=1$, we fix $M_x=M_y=20$;
the spatial error is negligible  since the spectral collocation method converges exponentially;
see, e.g., \cite[Theorem 4.4, \textsection{4.5.2}]{STW:2011}.

\begin{example}\label{ex:5.1}
The initial value and the forcing term are chosen such that the exact solution of equation \eqref{1.1} is
\begin{equation*}
\label{periodic}
u(x,t)=(t^4+1)(1-x^2)(1-y^2),\quad -1\leqslant x,y \leqslant 1,\  0\leqslant t \leqslant 1.
\end{equation*}

Here, we consider two cases of the adjacent time-step ratios $r_k$:

Case I: $r_{2k}=2,$ for $1\leq k\leq \frac{N}{2}$, and $r_{2k-1}=1/2,$ for $2\leq k\leq \frac{N}{2}$.

Case II: the arbitrary meshes with random time-steps $\tau_k=T\sigma_k/S$ for $1\leq k\leq N$, where
$S=\sum_{k=1}^N\sigma_k$ and $\sigma_k\in(0,1)$ are random numbers subject to the uniform distribution  \cite{CYZ:21,LZ}.
\end{example}


\begin{table}[!ht]
 \begin{center}
  \caption {The discrete $L^2$-norm errors and numerical convergence orders.}
\begin{tabular*}{\linewidth}{@{\extracolsep{\fill}}*{10}{ccccccc}}       \hline\hline
&&&$Case$ I\\
$N$ &$\varepsilon^2=0.16$ &Rate &$\varepsilon^2=0.36$ &Rate &$\max r_k$  &$\min r_k$\\  \hline
~20  &2.8069e-04  &       &1.9512e-04  &        &2     &1/2\\
~40  &3.5943e-05  &2.9652 &2.4542e-05  &2.9910  &2     &1/2\\
~80  &4.5427e-06  &2.9841 &3.0751e-06  &2.9965  &2     &1/2\\
~160 &5.7085e-07  &2.9924 &3.8478e-07  &2.9985  &2     &1/2\\   \hline\hline
&&&$Case$ II\\
$N$ &$\varepsilon^2=0.16$ &Rate &$\varepsilon^2=0.36$ &Rate &$\max r_k$  &$\min r_k$\\  \hline
~20  &3.2456e-04  &       &2.3570e-04  &        &44.3916     &0.0405\\
~40  &4.7387e-05  &2.7759 &3.0397e-05  &2.9550  &48.8928     &0.0121\\
~80  &6.2198e-06  &2.9295 &4.2227e-06  &2.8477  &48.8928     &0.0121\\
~160 &7.8189e-07  &2.9918 &5.0184e-07  &3.0729  &76.0331     &0.0121\\   \hline\hline
    \end{tabular*}\label{table:1}
  \end{center}
\end{table}

As observed in Table \ref{table:1}, the variable steps BDF3 scheme \eqref{1.7} achieves  the third-order accuracy, which agrees with Theorem \ref{Theorem:5.2}.
The improved condition $r_k\leq1.405$ is a sufficient condition for third-order convergence.

\bibliographystyle{amsplain}

\end{document}